\documentclass[11pt]{amsart}
\usepackage[margin=1in]{geometry}
\usepackage{hyperref,amssymb,amsfonts,amsthm,mathtools,enumerate,verbatim,
color,makecell}
\usepackage[alphabetic]{amsrefs}

\newtheorem{Theorem}{Theorem}[section]
\newtheorem{Proposition}[Theorem]{Proposition}
\newtheorem{Lemma}[Theorem]{Lemma}
\newtheorem{Corollary}[Theorem]{Corollary}
\newtheorem*{MainTheorem}{Theorem~\ref{thm-BigEquiv}}
\newtheorem*{Claim}{Claim}
\newtheorem{MClaim}{Claim}
\newtheorem{NClaim}{Claim}

\theoremstyle{definition}
\newtheorem{Definition}[Theorem]{Definition}

\newcommand{\rca}{\mathsf{RCA}_0}
\newcommand{\wkl}{\mathsf{WKL}_0}
\newcommand{\aca}{\mathsf{ACA}_0}
\newcommand{\rt}{\mathsf{RT}}

\DeclareMathOperator{\spec}{\mathrm{Spec}}

\DeclareMathOperator{\upper}{\mathcal{U}}
\DeclareMathOperator{\alex}{\mathcal{A}}

\newcommand{\andd}{\wedge}
\newcommand{\orr}{\vee}
\newcommand{\la}{\langle}
\newcommand{\ra}{\rangle}
\newcommand{\da}{{\downarrow}}
\newcommand{\ua}{{\uparrow}}
\newcommand{\imp}{\rightarrow}
\newcommand{\Imp}{\Rightarrow}
\newcommand{\biimp}{\leftrightarrow}
\newcommand{\Biimp}{\Leftrightarrow}
\newcommand{\Nb}{\mathbb{N}}

\newcommand{\Qb}{\mathbb{Q}}
\newcommand{\Rb}{\mathbb{R}}

\newcommand{\Pf}{\mathcal{P}_\mathrm{f}}
\newcommand{\Ps}{\mathcal{P}}

\newcommand{\MP}[1]{\mathcal{#1}}
\newcommand{\bs}[1]{{\mathbf #1}}

\title{Reverse mathematics, well-quasi-orders, and Noetherian spaces}

\author[Frittaion]{Emanuele Frittaion}
 \address{Mathematical Institute, Tohoku University, Japan}
\thanks{Emanuele Frittaion's research is supported by the Japan Society for the
Promotion of Science.}
\email{frittaion@math.tohoku.ac.jp}
\urladdr{http://www.math.tohoku.ac.jp/~frittaion/}

\author[Hendtlass]{Matt Hendtlass}
 \address{School of Mathematics and Statistics,
    University of Canterbury,
    Christchurch 8041,
    New Zealand}
\email{matthew.hendtlass@canterbury.ac.nz}

\author[Marcone]{Alberto Marcone}
 \address{Dipartimento di Matematica e Informatica,
    Universit\`{a} di Udine,
    viale delle Scienze 206,
    33100 Udine,
    Italy}
\email{alberto.marcone@uniud.it}
\urladdr{http://users.dimi.uniud.it/~alberto.marcone/}
\thanks{Alberto Marcone's research was supported by PRIN 2009 Grant ``Modelli
e Insiemi'' and PRIN 2012 Grant ``Logica, Modelli e Insiemi.''}

\author[Shafer]{Paul Shafer}
\address{Department of Mathematics\\
Ghent University\\
Krijgslaan 281 S22\\
B-9000 Ghent\\
Belgium}
\email{paul.shafer@ugent.be}
\urladdr{http://cage.ugent.be/~pshafer/}
\thanks{Paul Shafer is an FWO Pegasus Long Postdoctoral Fellow.}

\author[Van der Meeren]{Jeroen Van der Meeren}
\address{Department of Mathematics\\
Ghent University\\
Krijgslaan 281 S22\\
B-9000 Ghent\\
Belgium}
\email{jeroen.vandermeeren@ugent.be}
\urladdr{http://cage.ugent.be/~jvdm/}
\thanks{Jeroen Van der Meeren is an FWO Ph.D. Fellow.}

\date{November 27, 2015}

\begin{document}

\begin{abstract}
A quasi-order $Q$ induces two natural quasi-orders on $\Ps(Q)$, but if $Q$
is a well-quasi-order, then these quasi-orders need not necessarily be
well-quasi-orders. Nevertheless, Goubault-Larrecq in \cite{Gou-lics07}
showed that moving from a well-quasi-order $Q$ to the quasi-orders on
$\Ps(Q)$ preserves well-quasi-orderedness in a topological sense.
Specifically, Goubault-Larrecq proved that the upper topologies of the
induced quasi-orders on $\Ps(Q)$ are Noetherian, which means that they
contain no infinite strictly descending sequences of closed sets.  We
analyze various theorems of the form ``if $Q$ is a well-quasi-order then a
certain topology on (a subset of) $\Ps(Q)$ is Noetherian'' in the style of
reverse mathematics, proving that these theorems are equivalent to $\aca$
over $\rca$. To state these theorems in $\rca$ we introduce a new framework
for dealing with second-countable topological spaces.
\end{abstract}

\maketitle

\section{Introduction}
A topological space is \emph{Noetherian} if it satisfies the following
equivalent conditions.
\begin{itemize}
\item Every subspace is compact.

\item Every ascending sequence of open sets stabilizes: for every
    sequence $(G_n)_{n \in \Nb}$ of open sets such that $\forall n(G_n
  \subseteq G_{n+1})$, there is an $N$ such that $(\forall n > N)(G_n =
  G_N)$.

\item Every descending sequence of closed sets stabilizes: for
    every sequence $(F_n)_{n \in \Nb}$ of closed sets such that $\forall
  n(F_n \supseteq F_{n+1})$, there is an $N$ such that $(\forall n > N)(F_n
  = F_N)$.
\end{itemize}

The name `Noetherian space' comes from the typical example of a Noetherian
space, which is the Zariski topology on the spectrum of a Noetherian ring.
If $R$ is a commutative ring, let $\spec(R)$, the spectrum of $R$, denote the
set of prime ideals in $R$. The Zariski topology on $\spec(R)$ is the
topology whose closed sets are the sets of the form $\{P \in \spec(R) : I
\subseteq P\}$, where $I \subseteq R$ is an ideal. If the ring $R$ is
Noetherian, then $\spec(R)$ with the Zariski topology is a Noetherian space.

The present work, however, is not concerned with the connections between
Noetherian spaces and algebraic geometry but with the connections between
Noetherian spaces and the theory of well-quasi-orders. Goubault-Larrecq
in~\cite{Gou-lics07}, motivated by possible applications to verification
problems as explained in~\cite{Gou10}, provided several results demonstrating
that Noetherian spaces can be thought of as topological versions, or
generalizations, of well-quasi-orders. We analyze these theorems in the style
of reverse mathematics, proving that they are equivalent to $\aca$ over the
base theory $\rca$. As a byproduct of this analysis, we obtain elementary
proofs of Goubault-Larrecq's results which are much more direct than the
original category-theoretic arguments used in~\cite{Gou-lics07}. The logical
analysis of Noetherian spaces arising from Noetherian rings is ongoing work.

A quasi-order $Q$ induces various quasi-orders on $\Ps(Q)$, the power set of
$Q$, which in turn induce various topologies on $\Ps(Q)$. The theorems of
\cite{Gou-lics07} state that if a quasi-order $Q$ is in fact a
well-quasi-order, then several of the resulting topologies on $\Ps(Q)$ are
Noetherian. In order to state these results precisely, we must first
introduce the relevant definitions.

A \emph{quasi-order} is a pair $(Q,\leq_Q)$, where $Q$ is a set and $\leq_Q$
is a binary relation on $Q$ satisfying the reflexivity axiom $(\forall q \in
Q)(q \leq_Q q)$ and the transitivity axiom $(\forall p,q,r \in Q)((p \leq_Q q
\andd q \leq_Q r) \imp p \leq_Q r)$. For notational ease, we usually identify
$(Q,\leq_Q)$ and $Q$.  We write $p <_Q q$ when we have both $p \leq_Q q$ and
$q \nleq_Q p$, and we write $p \mid_Q q$ when we have both $p \nleq_Q q$ and
$q \nleq_Q p$.

If $E \in \Ps(Q)$, then $E\da = \{q \in Q : (\exists p \in E)(q \leq_Q p)\}$
denotes the downward closure of $E$ and $E\ua = \{q \in Q : (\exists p \in
E)(p \leq_Q q)\}$ denotes the upward closure of $E$. For $p \in Q$, we
usually write $p\da$ for $\{p\}\da$ and $p\ua$ for $\{p\}\ua$. A quasi-order
$Q$ induces the following quasi-orders $\leq_Q^\flat$ and $\leq_Q^\sharp$ on
$\Ps(Q)$. (We follow the notation of~\cite{Gou-lics07}. In other works, such
as~\cites{M2001,M2005}, `$\leq_Q^\flat$' is written as
`$\leq^\exists_\forall$' and `$\leq_Q^\sharp$' is written as
`$\leq^\forall_\exists$'.)

\begin{Definition}
Let $Q$ be a quasi-order. For $A, B \in \Ps(Q)$, define
\begin{itemize}
\item $A \leq_Q^\flat B$ if and only if $(\forall a \in A) (\exists b \in
    B) (a \leq_Q b)$, and
\item $A \leq_Q^\sharp B$ if and only if $(\forall b \in B) (\exists a \in
    A) (a \leq_Q b)$.
\end{itemize}
\end{Definition}
Notice that $A \leq_Q^\flat B$ is equivalent to $A \subseteq B\da$ and that
$A \leq_Q^\sharp B$ is equivalent to $B \subseteq A\ua $. We denote
$(\Ps(Q),\leq_Q^\flat)$ by $\Ps^\flat(Q)$ and $(\Ps(Q),\leq_Q^ \sharp)$ by
$\Ps^\sharp(Q)$: it is easy to check that these are indeed quasi-orders (they
are partial orders if and only if $Q$ is an antichain).

Both $\Ps^\flat(Q)$ and $\Ps^\sharp(Q)$ have been studied for a long time by
computer scientists. In this context, $\Ps^\flat(Q)$ is known as the
\emph{Hoare quasi-order}, and $\Ps^\sharp(Q)$ is known as the \emph{Smyth
quasi-order}. For example, these orders can be used to compare the different
executions of a non-deterministic computation (see e.g.\ \cite{Winskel} for
an early presentation).

A quasi-order can be topologized in several ways. We consider the Alexandroff
topology and the upper topology.

\begin{Definition}
Let $Q$ be a quasi-order.
\begin{itemize}
\item The \emph{Alexandroff topology} of $Q$ is the topology whose open
    sets are those of the form $E\ua$ for $E \subseteq Q$. The topological
    space consisting of $Q$ with its Alexandroff topology is denoted
    $\alex(Q)$.
\item The \emph{upper topology} of $Q$ is the topology whose basic open
    sets are those of the form $Q \setminus (E\da)$ for $E \subseteq Q$
    finite. The topological space consisting of $Q$ with its upper topology
    is denoted $\upper(Q)$.
\end{itemize}
\end{Definition}

Notice that, unless $Q$ is a partial order, neither $\alex(Q)$ nor
$\upper(Q)$ are $T_0$ spaces, and, unless $Q$ is an antichain, neither
$\alex(Q)$ nor $\upper(Q)$ are $T_1$ spaces. The order-theoretic
significances of $\alex(Q)$ and $\upper(Q)$ are that they are the finest and
coarsest topologies on $Q$ from which $\leq_Q$ can be recovered. Every
topological space $X$ induces a \emph{specialization quasi-order} on $X$
defined by $x \preceq y$ if and only if every open set that contains $x$ also
contains $y$. If $Q$ is a quasi-order, then $\alex(Q)$ (respectively
$\upper(Q)$) is the finest (respectively coarsest) topology on $Q$ for which
${\preceq} = {\leq_Q}$ (see for example \cite[Section~4.2]{GLBook}).

Finally, a \emph{well-quasi-order} (\emph{wqo}) is a quasi-order $Q$ that is
well-founded and has no infinite antichains. Equivalently, a quasi-order $Q$
is a wqo if for every function $f \colon \Nb \imp Q$, there are $m,n \in \Nb$
with $m < n$ such that $f(m) \leq_Q f(n)$.  It is easy to check that, for a
quasi-order $Q$, $Q$ is a wqo if and only if $\alex(Q)$ is Noetherian, in
which case $ \upper(Q)$ is also Noetherian.  In fact, in
Proposition~\ref{prop-WQOvsTopInRCA} we show that these facts are provable in
$\rca$.

For a quasi-order $Q$, let $\Pf(Q)$ denote the set of finite subsets of $Q$,
and let $\Pf^\flat(Q)$ and $\Pf^\sharp(Q)$ denote the respective restrictions
of $\Ps^\flat(Q)$ and $\Ps^\sharp(Q)$ to $\Pf(Q)$.
If $Q$ is a wqo, then $\Pf^\flat(Q)$ is also a wqo (see~\cite{ErdosRado}), but
$\Ps^\flat(Q)$, $\Ps^\sharp(Q)$, and
$\Pf^\sharp(Q)$ need not be wqo's. This can be seen by considering
Rado's example~\cite{RadoWPO}, the well-quasi-order $(R, \leq_R)$ where $R =
\{(i,j) \in \Nb \times \Nb : i < j\}$ and $(i,j) \leq_R (k,\ell)$ if $(i=k
\andd j \leq \ell) \vee j < k$ (see \cites{JanvcarWQO,M2001} for a complete
explanation).  Nevertheless, Goubault-Larrecq~\cite{Gou-lics07} proved that 
although passing from a wqo $Q$ to the quasi-orders $\Ps^\flat(Q)$, $\Ps^
\sharp(Q)$, and $\Pf^\sharp(Q)$ does not necessarily preserve
well-quasi-orderedness, it does preserve well-foundedness in the sense
that the upper topologies of $\Ps^\flat(Q)$ and $\Pf^\sharp(Q)$ are Noetherian.

\begin{Theorem}[\cite{Gou-lics07}]\label{thm-GL}
If $Q$ is a wqo, then $\upper(\Ps^\flat(Q))$ and $\upper(\Pf^\sharp(Q))$ are
Noetherian.
\end{Theorem}

Though Goubault-Larrecq explicitly proved Theorem~\ref{thm-GL} for
$\upper(\Ps^\flat(Q))$ and $\upper(\Pf^\sharp(Q))$ only,
it is also true that if $Q$ is a  wqo, then $\alex(\Pf^\flat(Q))$, 
$\upper(\Pf^\flat(Q))$, 
and $\upper(\Ps^\sharp(Q))$ are Noetherian as well.  For $\alex(\Pf^\flat(Q))$ and 
$\upper(\Pf^\flat(Q))$, this is because if $Q$ is a wqo, then $\Pf^\flat(Q)$ is also a 
wqo (the $\upper(\Pf^\flat(Q)) $ case also follows from Theorem~\ref{thm-GL}).  
The $\upper(\Ps^\sharp(Q))$ case follows from Theorem~\ref{thm-GL} because 
for every $A \in \Ps(Q)$ there is a $B \in \Pf(Q)$ that is equivalent to $A$ in the 
sense that $A \leq_Q^\sharp B$ and $B \leq_Q^\sharp A$.  Notice, however, that if 
$Q$ is a wqo, then $\alex(\Ps^\flat(Q))$, $\alex(\Ps^\sharp(Q))$, and $\alex(\Pf^
\sharp(Q))$ need not necessarily be Noetherian.   This is because $\Ps^\flat(Q)$, 
$\Ps^\sharp(Q)$, and $\Pf^\sharp(Q)$ need not necessarily be wqo's, and a
quasi-order's Alexandroff topology is Noetherian if and only if the quasi-order
itself is a wqo.

Concerning $\Ps^\flat(Q)$ and $\Ps^\sharp(Q)$, we wish to remark that
Nash-Williams~\cite{NashWilliamsBQO} strengthened well-quasi-orders to
better-quasi-orders (bqo's), an ingenious insight that led to a rich theory,
including Laver's proof of Fra\"iss\'e's conjecture in~\cite{Laver}.  A few
years later, Pouzet~\cite{Pouzet} introduced a hierarchy of notions
intermediate between wqo and bqo by defining the $\alpha$-well-quasi-orders
($\alpha$-wqo's) for each countable ordinal $\alpha$.  The $\omega$-wqo's are
exactly the wqo's, and the larger $\alpha$ is, the closer the notion of
$\alpha$-wqo is to the notion of bqo.  Indeed, $Q$ is a bqo if and only if
$Q$ is an $\alpha$-wqo for every $\alpha<\omega_1$.  By imposing these
stronger conditions on $Q$, we may ensure that $\Ps^\flat(Q)$ and
$\Ps^\sharp(Q)$ are wqo's.

\begin{Theorem}[see~\cite{M2001} for a complete discussion and further
results]{\ }
\begin{itemize}
\item If $Q$ is a bqo, then $\Ps^\flat(Q)$, $\Ps^\sharp(Q)$,
    $\Pf^\flat(Q)$, and $\Pf^\sharp(Q)$ are all bqo's.
\item If $Q$ is a $\omega^2$-wqo, then $\Ps^\flat(Q)$, $\Ps^\sharp(Q)$,
    $\Pf^\flat(Q)$, and $\Pf^\sharp(Q)$ are all wqo's.
\end{itemize}
\end{Theorem}

The purpose of this work is to study Theorem~\ref{thm-GL} and related
statements from the viewpoint of reverse mathematics. Our main result is the
following.

\begin{MainTheorem}
The following are equivalent over $\rca$.
\begin{enumerate}[(i)]
\item $\aca$.
\item If $Q$ is a wqo, then $\alex(\Pf^\flat(Q))$ is Noetherian.
\item If $Q$ is a wqo, then $\upper(\Pf^\flat(Q))$ is Noetherian.
\item If $Q$ is a wqo, then $\upper(\Pf^\sharp(Q))$ is Noetherian.
\item If $Q$ is a wqo, then $\upper(\Ps^\flat(Q))$ is Noetherian.
\item If $Q$ is a wqo, then $\upper(\Ps^\sharp(Q))$ is Noetherian.
\end{enumerate}
\end{MainTheorem}

The following table summarizes the logical strengths of implications such as ``if 
$Q$ is a wqo, then $\Pf^\flat(Q)$ is a wqo'' and  ``if $Q$ is a wqo, then $
\upper(\Ps^\flat(Q))$ is Noetherian.''  The entry `false' indicates that the 
corresponding implication is false, as witnessed by Rado's example.  The entry `$
\rca$' indicates that the corresponding implication is provable in $\rca$.  The entry 
`$\aca$' indicates that the corresponding implication is equivalent to $\aca$ over $
\rca$.  The entry `$\leq \aca$' indicates that the corresponding implication is 
provable in $\aca$ but a reversal is not yet known.  The table also provides 
references for the true implications.

\begin{center}
\begin{tabular}{|l|c|c|c|c|}
\hline
	& \makecell{$\Ps^\flat(Q)$}
	& \makecell{$\Ps^\sharp(Q)$}
	& \makecell{$\Pf^\flat(Q)$}
	& \makecell{$\Pf^\sharp(Q)$} \\ 
\hline
\makecell{$Q$ wqo $\Imp$ $\bullet$ wqo}
	& \makecell{false}
	& \makecell{false}
	& \makecell{$\aca$ \\
		\cite[5.10]{M2005} \\ Thm~\ref{thm-FlatWQOinACA}}
	& \makecell{false} \\
\hline
\makecell{$Q$ bqo $\Imp$ $\bullet$ bqo}
	& \makecell{$\leq \aca$ \\ \cite[5.4]{M2005}}
	& \makecell{$\rca$ \\ \cite[5.6]{M2005}}
	& \makecell{$\rca$ \\ \cite[5.4]{M2005}}
	& \makecell{$\rca$ \\ \cite[5.4]{M2005}} \\
\hline
\makecell{$Q$ wqo $\Imp$ $\alex(\bullet)$ Noeth.}
	& \makecell{false}
	& \makecell{false}
	& \makecell{$\aca$ \\ Thm~\ref{thm-BigEquiv}}
	& \makecell{false}  \\ 
\hline
\makecell{$Q$ wqo $\Imp$ $\upper(\bullet)$ Noeth.}
	& \makecell{$\aca$ \\ Thm~\ref{thm-BigEquiv}}
	& \makecell{$\aca$ \\ Thm~\ref{thm-BigEquiv}}
	& \makecell{$\aca$ \\ Thm~\ref{thm-BigEquiv}}
	& \makecell{$\aca$ \\ Thm~\ref{thm-BigEquiv}} \\
\hline
\end{tabular}
\end{center}

If $Q$ is a countable quasi-order, then $\Pf(Q)$ is also countable and hence
easy to manage in second-order arithmetic. The spaces $\alex(\Pf^\flat(Q))$,
$\upper(\Pf^\flat(Q))$, and $\upper(\Pf^\sharp(Q))$ fit very nicely into
Dorais's framework of countable second-countable spaces in second-order
arithmetic \cite{Dorais:2011uv}, and so we consider the equivalence of items
(i)--(iv) in Theorem~\ref{thm-BigEquiv} as not only contributing to the
reverse mathematics of wqo's but also as a proof-of-concept example of the
usefulness of Dorais's framework.

On the other hand, if $Q$ is infinite, then $\Ps(Q)$ is uncountable and
hence neither it nor the basic open sets of $\upper(\Ps^\flat(Q))$ and
$\upper(\Ps^\sharp(Q))$ exist as sets in second-order arithmetic. Thus for
items~(v) and (vi) of Theorem~\ref{thm-BigEquiv}, we code
$\upper(\Ps^\flat(Q))$ and $\upper(\Ps^\sharp(Q))$ using a scheme that is
broadly similar to the usual coding of complete separable metric spaces in
second-order arithmetic (as detailed in \cite[Section~II.5]{Simpson:2009vv},
for example). This scheme can be adapted to deal with quite general
second-countable topological spaces, including the countably based MF spaces
of \cite{Mummert06}.

To prove the reversals of Theorem~\ref{thm-BigEquiv}, we isolate a way of
constructing recursive partial orders such that every sequence witnessing
that such a partial order is not a wqo computes $0'$. These partial orders
generalize the recursive linear order of type $\omega+\omega^*$ used in
\cites{MarSho11,FriMar12,FriMar14} in which every sequence witnessing that
the linear order is not a well-order computes $0'$. The construction is
introduced in Definition~\ref{def-Xi}, and its main property is proved in
Lemma~\ref{lem-ACAreversal}.

The plan of the paper is as follows. In Section~\ref{sec-background} we give
some background concerning reverse mathematics in general, the reverse
mathematics of well-quasi-orders, and Dorais's coding of countable
second-countable topological spaces. Section~\ref{sec-Noeth} covers the
details of expressing the notion of Noetherian space in second-order
arithmetic, both in the countable second-countable case and in the
uncountable case. In this section we also show that $\aca$ proves statements
(ii)-(vi) of Theorem~\ref{thm-BigEquiv}. The reversals of these implications
are proved in Section~\ref{sec-reversals}, using the construction mentioned
in the previous paragraph.

\section{Background}\label{sec-background}

\subsection{Reverse mathematics}

Reverse mathematics is a foundational program introduced by Friedman
\cite{Friedman} with the goal of classifying the theorems of ordinary
mathematics by their proof-theoretic strengths. Theorem $\varphi$ is
considered stronger than theorem $\psi$ if $\varphi$ requires stronger axioms
to prove than $\psi$ does or, equivalently, if $\varphi$ implies $\psi$ but
not conversely over some fixed weak base theory. The usual setting for
reverse mathematics is second-order arithmetic. The
language of second-order arithmetic is a two-sorted language, with
first-order variables (intended to range over natural numbers) and
second-order variables (intended to range over sets of natural numbers), and
the membership relation to connect the two sorts. In this setting a
remarkable phenomenon is that a natural theorem $\varphi$ is most often
equivalent to some well-known theory $T$ over the base theory $B$. This means
that $T \vdash \varphi$ and that $B + \varphi \vdash \psi$ for every $\psi
\in T$. The proofs of the axioms of $T$ from $B + \varphi$ is called a
\emph{reversal}, from which `reverse mathematics' gets its name. This article
is only concerned with the standard base theory $\rca$ and the theory $\aca$,
so we give only the definitions of these theories and refer the reader
to~\cite{Simpson:2009vv} for a comprehensive treatment of the reverse
mathematics program.

The axioms of $\rca$ are: a first-order sentence expressing that $\Nb$ is a
discretely ordered commutative semi-ring with identity; the
\emph{$\Sigma^0_1$ induction scheme}, which consists of the universal
closures (by both first- and second-order quantifiers) of all formulas of the
form
\begin{align*}
[\varphi(0) \andd \forall n(\varphi(n) \imp \varphi(n+1))] \imp \forall n \varphi(n),
\end{align*}
where $\varphi$ is $\Sigma^0_1$; and the \emph{$\Delta^0_1$ comprehension
scheme}, which consists of the universal closures (by both first- and
second-order quantifiers) of all formulas of the form
\begin{align*}
\forall n (\varphi(n) \biimp \psi(n)) \imp \exists X \forall n(n \in X \biimp \varphi(n)),
\end{align*}
where $\varphi$ is $\Sigma^0_1$, $\psi$ is $\Pi^0_1$, and $X$ is not free in
$\varphi$.

The system $\rca$ is taken as the standard base system in reverse
mathematics. The name `$\rca$' stands for `recursive comprehension axiom',
which refers to the $\Delta^0_1$ comprehension scheme because a set $X$ is
$\Delta^0_1$ in a set $Y$ if and only if $X$ is recursive in $Y$. Thus in
$\rca$, to define a set by comprehension, one must compute that set from an
existing set. For this reason, we think of $\rca$ as capturing what might be
called `recursive mathematics' or `effective mathematics'. The subscript
`$0$' in `$\rca$' refers to the fact that induction in $\rca$ is limited to
$\Sigma^0_1$ formulas (and to $\Pi^0_1$ formulas because $\rca$ proves the
$\Pi^0_1$ induction scheme; see \cite[Corollary~II.3.10]{Simpson:2009vv}).
Despite being a weak system, several interesting and familiar theorems are
provable in $\rca$, such as the intermediate value theorem and the fact that
every field has an algebraic closure (though $\rca$ does \emph{not} suffice
to prove that algebraic closures are unique). See
\cite[Chapter~II]{Simpson:2009vv} for more about $\rca$.

The axioms of $\aca$ are a first-order sentence expressing that $\Nb$ is a
discretely ordered commutative semi-ring with identity; the \emph{induction
axiom}
\begin{align*}
\forall X [[0 \in X \andd \forall n(n \in X \imp n+1 \in X)] \imp \forall n (n \in X)];
\end{align*}
and the \emph{arithmetical comprehension scheme}, which consists of the
universal closures (by both first- and second-order quantifiers) of all
formulas of the form
\begin{align*}
\exists X \forall n(n \in X \biimp \varphi(n)),
\end{align*}
where $\varphi$ is an arithmetical formula in which $X$ is not free.  Equivalently, $
\aca$ may be obtained by adding the arithmetical comprehension scheme to the
axioms of $\rca$.

The name `$\aca$' stands for `arithmetical comprehension axiom', which refers
to the arithmetical comprehension scheme. The subscript `$0$' in `$\aca$'
refers to the fact that induction in $\aca$ is essentially limited to
arithmetical formulas, which is what can be derived from the induction axiom
and the arithmetical comprehension scheme. In terms of computability, $\aca$
can be characterized by adding the statement ``for every set $X$, the Turing
jump of $X$ exists'' to $\rca$. Many familiar theorems are equivalent to
$\aca$ over $\rca$, such as the Bolzano-Weierstra{\ss} theorem, the fact that
every vector space has a basis, the fact that every commutative ring has a
maximal ideal (the existence of prime ideals is weaker), K\"onig's lemma, and
Ramsey's theorem for $k$-tuples for any fixed $k \geq 3$ (Ramsey's theorem
for pairs is weaker, and Ramsey's theorem for arbitrary tuples is stronger).
See \cite[Chapter~III]{Simpson:2009vv} for more about $\aca$.

A common strategy for proving that a theorem reverses to $\aca$ over $\rca$
is to take advantage of the following lemma, which states that $\aca$ is
equivalent over $\rca$ to the statement that every injection has a range.

\begin{Lemma}[{\cite[Lemma~III.1.3]{Simpson:2009vv}}]\label{lem-ACAinjection}
The following are equivalent over $\rca$.
\begin{itemize}
\item[(i)] $\aca$.
\item[(ii)] If $f \colon \Nb \imp \Nb$ is an injection, then there is a set
    $X$ such that
\begin{align*}
\forall n(n \in X \biimp \exists m(f(m) = n)).
\end{align*}
\end{itemize}
\end{Lemma}

\subsection{Well-quasi-orders in second-order arithmetic}

The reverse mathematics of wqo and bqo theory is a vibrant area with many
results and many open problems, and we refer the reader to~\cite{M2005} for a
thorough introduction. Here we simply present the basic information needed
for the work at hand.

In $\rca$ we can easily give the definition of quasi-order made in the
introduction. In $\rca$, the official definition of a well-quasi-order is the
following.

\begin{Definition}[$\rca$]\label{def-wqo}
A \emph{well-quasi-order} (\emph{wqo}) is a quasi-order $Q$ such that for
every function $f \colon \Nb \imp Q$, there are $m,n \in \Nb$ with $m < n$
such that $f(m) \leq_Q f(n)$.
\end{Definition}

We usually think of a function $f \colon \Nb \imp Q$ as a sequence $(q_n)_{n
\in \Nb}$ of elements of $Q$, in which case Definition~\ref{def-wqo} states
that a quasi-order $Q$ is a wqo if for every such sequence there are $m,n \in
\Nb$ with $m < n$ such that $q_m \leq_Q q_n$. Thus $Q$ is not a wqo if and
only if there is an infinite so-called \emph{bad sequence} $(q_n)_{n \in
\Nb}$ such that $\forall m \forall n(m < n \imp q_m \nleq_Q q_n)$. For
convenience, we also define a finite sequence $(q_n)_{n<k}$ to be bad if
$\forall m \forall n(m<n<k \imp q_m \nleq_Q q_n)$.

Several of the well-known classically equivalent definitions of
well-quasi-order are not equivalent over $\rca$. For example, $\rca$ proves
that if $Q$ is a wqo according to Definition~\ref{def-wqo}, then $Q$ has no
infinite strictly descending chains and no infinite antichains~\cite{M2005}.
However, the reverse implication, that a quasi-order with no infinite
strictly descending chains and no infinite antichains is a wqo according to
Definition~\ref{def-wqo}, is equivalent to $\mathsf{CAC}$ over $\rca$, where
$\mathsf{CAC}$ states that every infinite partial order has an infinite chain
or an infinite antichain \cites{ChoMarSol04,FrittaionThesis}.  Thus the
equivalence of these two definitions of wqo is provable in $\rca +
\mathsf{CAC}$ but not in $\rca$.

By using the usual coding of finite subsets of $\Nb$ as elements of $\Nb$,
one readily sees that $\rca$ proves that if $Q$ is a quasi-order, then
$\Pf(Q)$,
\begin{align*}
\leq_Q^\flat = \{(\bs a, \bs b) \in \Pf(Q) \times \Pf(Q) :
(\forall a \in \bs a)(\exists b \in \bs b)(a \leq_Q b)\},
\end{align*}
and
\begin{align*}
\leq_Q^\sharp = \{(\bs a, \bs b) \in \Pf(Q) \times \Pf(Q) : (\forall b \in \bs b)(\exists 
a
\in \bs a)(a \leq_Q b)\}
\end{align*}
all exist as sets. The proof that if $Q$ is a
quasi-order then $\Pf^\flat(Q)$ and $\Pf^\sharp(Q)$ are both quasi-orders is
also straightforward in $\rca$.

A little care must be taken to describe $\Ps^\flat(Q)$ and $\Ps^\sharp(Q)$ in
$\rca$. First, if $Q$ is infinite, then $\Ps(Q)$ is of course too big to
exist as a set in any subsystem of second-order arithmetic. Second, if $Q$ is
a quasi-order and $E \subseteq Q$, then $\rca$ proves that $E\da$ and $E\ua$
exist as sets when $E$ is finite, but in general $\aca$ is required to prove
that $E\da$ and $E\ua$ exist as sets when $E$ is infinite. Thus when working
in $\rca$, `$A \leq^\flat_Q B$' and `$A \leq^\sharp_Q B$' must be interpreted
by their respective defining formulas `$(\forall a \in A)(\exists b \in B)(a
\leq_Q b)$' and `$(\forall b \in B)(\exists a \in A)(a \leq_Q b)$.' Under
this interpretation, in $\rca$ one can prove that $A \leq_Q^\flat B
\leq_Q^\flat C \imp A \leq_Q^\flat C$ for all $A,B,C \subseteq Q$, one can
work with sequences $(A_n)_{n \in \Nb}$ of subsets of $Q$, and one can
consider whether or not there are $m, n \in \Nb$ with $m < n$ such that $A_m
\leq_Q^\flat A_n$. Using this approach, Marcone has shown the following
theorem.

\begin{Theorem}[{\cite[Theorem~5.4 and Theorem~5.6]{M2005}}]{\ }
\begin{itemize}
\item $\rca$ proves that if $Q$ is a bqo, then $\Pf^\flat(Q)$, $\Pf^\sharp(Q)$, and $
\Ps^\sharp(Q)$ are all bqo's.

\item $\aca$ proves that if $Q$ is a bqo, then $\Ps^\flat(Q)$ is a bqo.
\end{itemize}
\end{Theorem}
The reversal for the second item in the above theorem remains open.

Theorem~5.10 of \cite{M2005} states that $\aca$ is equivalent to the
statement ``if $Q$ is a wqo then $\Pf^\flat(Q)$ is a wqo'' over $\rca +
\rt^2_2$, where $\rt^2_2$ is Ramsey's theorem for pairs and two colors. In
the proof of the reversal of this theorem, $\rt^2_2$ is only used to prove
that $Q \times R$ is a wqo (where $(q_0,r_0) \leq_{Q \times R} (q_1,r_1)$ if
and only if $q_0 \leq_Q q_1$ and $r_0 \leq_R r_1$) whenever $Q$ and $R$ are
wqo's. Notice that by \cite[Corollary 4.7]{ChoMarSol04} $\rca$ and even the
stronger system $\wkl$ do not suffice to prove this statement. We eliminate
$\rt^2_2$ from the reversal, thereby improving the result of \cite{M2005},
via the following lemma.

\begin{Lemma}[$\rca$]\label{lem-FlatIsACAHelper}
Suppose that $\Pf^\flat(Q)$ is a wqo whenever $Q$ is a wqo. Then $Q \times R$
is a wqo whenever $Q$ and $R$ are wqo's.
\end{Lemma}

\begin{proof}
Let $Q$ and $R$ be wqo's, and let $((q_n,r_n))_{n \in \Nb}$ be a sequence of
elements from $Q \times R$. We need to find $m$ and $n$ in $\Nb$ with $m < n$
such that $(q_m,r_m) \leq_{Q \times R} (q_n,r_n)$. To this end, let $Q \oplus
R$ be the disjoint sum of $Q$ and $R$, where $Q \oplus R = (Q \times \{0\})
\cup (R \times \{1\})$ and $(x,i) \leq_{Q \oplus R} (y,j)$ if and only if
($i=j=0 \andd x \leq_Q y) \orr (i=j=1 \andd x \leq_R y)$. It is easy to see
that $Q \oplus R$ is a quasi-order, and by \cite[Lemma 5.13]{M2005} it is a
wqo. By hypothesis, $\Pf^\flat(Q \oplus R)$ is also a wqo. Consider now the
sequence $(\{(q_n,0), (r_n,1)\})_{n \in \Nb}$ of elements of $\Pf^\flat(Q
\oplus R)$, and let $m$ and $n$ in $\Nb$ be such that $m < n$ and $\{(q_m,0),
(r_m,1)\} \leq_{Q \oplus R}^\flat \{(q_n,0), (r_n,1)\}$. It must be that $q_m
\leq_Q q_n$ and $r_m \leq_R r_n$, so we have our desired $m < n$ such that
$(q_m,r_m) \leq_{Q \times R} (q_n,r_n)$.
\end{proof}

\begin{Theorem}\label{thm-FlatWQOinACA}
The following are equivalent over $\rca$.
\begin{itemize}
\item[(i)] $\aca$.
\item[(ii)] If $Q$ is a wqo, then $\Pf^\flat(Q)$ is a wqo.
\end{itemize}
\end{Theorem}

\begin{proof}
Use the proof of \cite[Theorem~5.10]{M2005}, but, in the reversal, prove that $L
\times \omega$ is a wqo using Lemma~\ref{lem-FlatIsACAHelper} instead of $
\rt^2_2$.
\end{proof}

Theorem~\ref{thm-FlatReversal} below provides a different proof of the
(ii)$\Imp$(i) implication in Theorem~\ref{thm-FlatWQOinACA}.

\subsection{Countable second-countable topological spaces in second-order
arithmetic}

A topological space is second-countable if it has a a countable base. Dorais
in~\cite{Dorais:2011uv} provides the appropriate definitions for working with
countable second-countable spaces in $\rca$. We situate our work in his
framework.

\begin{Definition}[$\rca$; {\cite[Definition~2.1]{Dorais:2011uv}}]
A \emph{base} for a topology on a set $X$ is an indexed sequence $\MP U =
(U_i)_{i \in I}$ of subsets of $X$ together with a function $k \colon X
\times I \times I \imp I$ such that the following properties hold.
\begin{itemize}
\item If $x \in X$, then $x \in U_i$ for some $i \in I$.
\item If $x \in U_i \cap U_j$, then $x \in U_{k(x,i,j)} \subseteq U_i \cap
    U_j$.
\end{itemize}
\end{Definition}

\begin{Definition}[$\rca$; {\cite[Definition~2.2]{Dorais:2011uv}}]
\label{def-CSCspace}
A \emph{countable second-countable space} is a triple $(X, \MP U, k)$ where
$\MP U = (U_i)_{i \in I}$ and $k \colon X \times I \times I \imp I$ form a
base for a topology on the set $X$.
\end{Definition}

Subsets of countable second-countable spaces produce subspaces in a natural
way.

\begin{Definition}[$\rca$; {\cite[Definition~2.9]{Dorais:2011uv}}]
\label{def-subspace}
Let $(X, \MP U, k)$ be a countable second-countable space with $\MP U =
(U_i)_{i \in I}$. If $X' \subseteq X$, then the corresponding \emph{subspace}
$(X', \MP{U}', k')$ is defined by $U_i' = U_i \cap X'$ for all $i \in I$ and
$k' = k \restriction (X' \times I \times I)$.
\end{Definition}

Let $(X, \MP U, k)$ be a countable second-countable space, and recall that
$\Pf(I)$ denotes the set of finite subsets of $I$. Every function $h \colon
\Nb \imp \Pf(I)$ codes a so-called \emph{effectively open} set, the idea
being that $h$ enumerates (sets of) indices of basic open sets whose union is
the open set being coded. Explicitly, $h$ is a code for the open set $G_h =
\bigcup_{n \in \Nb}\bigcup_{i \in h(n)}U_i$. Of course $\rca$ does not prove
that such unions exist in general, so we must interpret the statement ``$x$
is in the effectively open set coded by $h$'' as the formula `$(\exists
n)(\exists i \in h(n))(x \in U_i)$'. To simplify notation, we abbreviate this
formula by `$x \in G_h$' or by `$x \in \bigcup_{n \in \Nb}\bigcup_{i \in
h(n)}U_i$'. Similarly, we also interpret $h$ as coding the \emph{effectively
closed} set $F_h = X \setminus G_h = \bigcap_{n \in \Nb}\bigcap_{i \in
h(n)}(X \setminus U_i)$. The reason for coding open sets by functions $\Nb
\imp \Pf(I)$ rather than by functions $\Nb \imp I$ is that the coding by
functions $\Nb \imp \Pf(I)$ allows for a natural coding of the empty set via
the function that is constantly $\emptyset$. Otherwise we would need to
enforce that $U_i = \emptyset$ for some $i \in I$ for there to be a code for
the empty set as an effectively open set.

\begin{Definition}[$\rca$; {\cite[Definition~3.1]{Dorais:2011uv}}]\label{def-compact}
Let $(X, \MP U, k)$ be a countable second-countable space with $\MP U =
(U_i)_{i \in I}$. We say that $(X, \MP U, k)$ is \emph{compact} if for every
$h \colon \Nb \imp \Pf(I)$ such that $X = \bigcup_{n \in \Nb}\bigcup_{i \in
h(n)}U_i$ (i.e., such that $(\forall x \in X)(\exists n \in \Nb)(\exists i
\in h(n))(x \in U_i)$), there is an $N \in \Nb$ such that $X = \bigcup_{n <
N}\bigcup_{i \in h(n)}U_i$.
\end{Definition}

Dorais's \cite[Proposition~3.2]{Dorais:2011uv} expresses that this definition
of compactness does not depend on the choice of base $(\MP U, k)$ for the
topology on $X$. For matters of convenience, Definition~\ref{def-compact}
defines compactness in terms of covers by basic open sets. It is equivalent
to define compactness in terms of covers by arbitrary open sets. Let $(X, \MP
U, k)$ be a countable second-countable space. A \emph{sequence of effectively
open sets} in $(X, \MP U, k)$ is a function $g \colon \Nb \times \Nb \imp
\Pf(I)$ thought of as coding the sequence $(G_n)_{n \in \Nb}$, where each
$G_n$ is $G_{g(n,\cdot)} = \bigcup_{m \in \Nb}\bigcup_{i \in g(n,m)}U_i$.
Similarly, a \emph{sequence of effectively closed sets} in $(X, \MP U, k)$ is
a function $g \colon \Nb \times \Nb \imp \Pf(I)$ thought of as coding the
sequence $(F_n)_{n \in \Nb}$, where each $F_n$ is $F_{g(n,\cdot)} =
\bigcap_{m \in \Nb}\bigcap_{i \in g(n,m)}X \setminus U_i$. $\rca$ proves that
a countable second-countable space $(X, \MP U, k)$ is compact if and only if
for every sequence $(G_n)_{n \in \Nb}$ of effectively open sets such that $X
= \bigcup_{n \in \Nb}G_n$, there is an $N \in \Nb$ such that $X = \bigcup_{n
< N}G_n$.

\section{Noetherian spaces in second-order arithmetic}\label{sec-Noeth}

\subsection{Countable second-countable spaces}

Let $(X, \MP U, k)$ be a countable second-countable space, and let $G_h$ be
an effectively open set. One is tempted to define compactness for $G_h$ via
Definition~\ref{def-subspace} and Definition~\ref{def-compact} by saying that
the subspace corresponding to $G_h$ is compact. However, $G_h$ is a coded
object, and $\rca$ need not in general prove that it exists as a set, and so
Definition~\ref{def-subspace} need not apply.\footnote{However, if $Y$ is a
$\Sigma^0_1$ subset of $X$ and $f \colon \Nb \imp Y$ is an enumeration of
$Y$, then $(\Nb, \MP V, \ell)$, where $V_i = f^{-1}(U_i)$ and $\ell(n,i,j) =
k(f(n),i,j)$ is a countable second-countable space that is essentially a
homeomorphic copy of the subspace of $(X, \MP U, k)$ corresponding to $Y$.}
We simply extend Definition~\ref{def-compact} as follows.

\begin{Definition}[$\rca$]\label{def-opencompact}
Let $(X, \MP U, k)$ be a countable second-countable space with $\MP U =
(U_i)_{i \in I}$.  An effectively open set $G_h$ is \emph{compact} if for every
function $f \colon \Nb \imp \Pf(I)$ with $G_h \subseteq \bigcup_{n \in \Nb}
\bigcup_{i \in f(n)}U_i$, there is an $N \in \Nb$ such that $G_h \subseteq
\bigcup_{n < N}\bigcup_{i \in f(n)}U_i$.
\end{Definition}

Note that, \emph{a posteriori}, if $G_h$ is a compact effectively open set,
then $G_h = \bigcup_{n < N}\bigcup_{i \in h(n)}U_i$ for some $N \in \Nb$, so
$\rca$ does indeed prove that it exists as a set and that the corresponding
subspace is compact.

Now we show that the equivalent definitions of Noetherian space are
equivalent over $\rca$.

\begin{Proposition}[$\rca$]\label{prop-NoethEquiv}
For a countable second-countable space $(X, \MP U, k)$, the following
statements are equivalent.
\begin{itemize}
\item[(i)] Every effectively open set is compact.

\item[(ii)] For every effectively open set $G_h$, there is an $N \in \Nb$
    such that $G_h = \bigcup_{n < N}\bigcup_{i \in h(n)}U_i$.

\item[(iii)] Every subspace is compact.

\item[(iv)] For every sequence $(G_n)_{n \in \Nb}$ of effectively open sets
    such that $\forall n(G_n \subseteq G_{n+1})$, there is an $N$ such that
    $(\forall n > N)(G_n = G_N)$.

\item[(v)] For every sequence $(F_n)_{n \in \Nb}$ of effectively closed
    sets such that $\forall n(F_n \supseteq F_{n+1})$, there is an $N$ such
    that $(\forall n > N)(F_n = F_N)$.
\end{itemize}
\end{Proposition}

\begin{proof}
The proof that (i), (ii), (iv), and (v) are equivalent is a simple exercise
in chasing the definitions. Likewise, it is easy to see that each of (i),
(ii), (iv), and (v) implies (iii). That (iii) implies the others requires
proof. We show that (iii)$\Imp$(ii). Let $(X, \MP U, k)$ be a countable
second-countable space, and let $G_h$ be an effectively open set. We would
like to apply (iii) to the subspace $G_h$, but this cannot be done in $\rca$
because $G_h$ need not exist as a set.

Assume for a contradiction that no $N$ satisfies $G_h = \bigcup_{n <
N}\bigcup_{i \in h(n)}U_i$, and fix an enumeration $g \colon \Nb \imp X$ of
the elements of $G_h$ (which is possible because $G_h$ has a $\Sigma^0_1$
definition). We recursively define an injection $f \colon \Nb \imp X$.
Assuming we defined $f(m)$ for $m<n$, let
\begin{align*}
F_n = \{f(m) : m < n\} \cup \bigcup_{m < n}\bigcup_{i \in h(m)}U_i,
\end{align*}
and define $f(n) = g(k)$ where $k$ is least such that $g(k) \notin F_n$. Such
a $k$ exists because otherwise $G_h \subseteq F_n$ is the union of finitely
many of the sets $\bigcup_{i \in h(m)}U_i$.

Let $X' \subseteq X$ be an infinite set such that $(\forall x \in X')(\exists
n \in \Nb)(f(n) = x)$. (It is well-known and easy to show that $\rca$ proves
that if $f \colon \Nb \imp X$ is an injection, then the range of $f$ is
infinite and there is an infinite set $X'$ of elements in the range of $f$.
This is a formalization of the fact that every infinite r.e.\ set contains an
infinite recursive subset.) By (iii), the subspace $(X', \MP{U}', k')$ (using
the notation of Definition~\ref{def-subspace}) is compact, so there is an $N
\in \Nb$ such that $X' = \bigcup_{m < N}\bigcup_{i \in h(m)}(U_i \cap X')$.
Now pick $n>N$ such that $f(n) \in X'$. We have the contradiction that both
$f(n) \in \bigcup_{m < N}\bigcup_{i \in h(m)}U_i$ by the choice of $N$ and
$f(n) \notin \bigcup_{m < N}\bigcup_{i \in h(m)}U_i$ by the definition of
$f$.
\end{proof}

\begin{Definition}[$\rca$]
A countable second-countable space is \emph{Noetherian} if it satisfies any
of the equivalent conditions from Proposition~\ref{prop-NoethEquiv}.
\end{Definition}

We can make any quasi-order a countable second-countable space by giving it
either the Alexandroff topology or the upper topology.

\begin{Definition}[$\rca$]
Let $Q$ be a quasi-order.
\begin{itemize}
\item A base for the Alexandroff topology on $Q$ is given by $\MP U =
    (U_q)_{q \in Q}$, where $U_q = q\ua$ for each $q \in Q$, and $k(q, p,
    r) = q$.  Let $\alex(Q)$ denote the countable second-countable space $(Q,
    \MP U, k)$.

\item A base for the upper topology on $Q$ is given by $\MP V = (V_{\bs
    i})_{\bs i \in \Pf(Q)}$, where $V_{\bs i} = Q \setminus (\bs{i}\da)$
    for each $\bs i \in \Pf(Q)$, and $\ell(q, \bs i, \bs j) = \bs i \cup
    \bs j$.  Let $\upper(Q)$ denote the countable second-countable space $(Q,
    \MP V, \ell)$.
\end{itemize}
\end{Definition}

That a quasi-order's Alexandroff topology is finer than its upper topology
can be made precise via the following definition.

\begin{Definition}[$\rca$]\label{def-EffectivelyFiner}
Let $X$ be a set, and let $(\MP U = (U_i)_{i \in I}, k)$ and $(\MP V =
(V_j)_{j \in J}, \ell)$ be two bases for topologies on $X$. We say that $(X, \MP U,
k)$ is \emph{effectively finer} than $(X, \MP V, \ell)$ and that $(X, \MP V, \ell)$
is \emph{effectively coarser} than $(X, \MP U, k)$ if there is a function $f
\colon J \times \Nb \imp \Pf(I)$ such that
\begin{align*}
(\forall j \in J)(\forall x \in
X)(x \in V_j \biimp (\exists n \in \Nb)(\exists i \in f(j,n))(x \in U_i)).
\end{align*}
\end{Definition}

Essentially, $(X, \MP U, k)$ is effectively finer than $(X, \MP V, \ell)$ if
there is a sequence of sets $(G_j)_{j \in J}$ indexed by $J$ and effectively
open in $(X, \MP U, k)$ such that $(\forall j \in J)(G_j = V_j)$. It follows
that every effectively open set in $(X, \MP V, \ell)$ is effectively open in
$(X, \MP U, k)$, which leads to the following proposition.

\begin{Proposition}[$\rca$]\label{prop-finer}
Let $(X, \MP U, k)$ and $(X, \MP V, \ell)$ be countable second-countable
spaces with $(X, \MP U, k)$ effectively finer than $(X, \MP V, \ell)$.
\begin{itemize}
\item If $(X, \MP U, k)$ is compact, then $(X, \MP V, \ell)$ is compact.
\item If $(X, \MP U, k)$ is Noetherian, then $(X, \MP V, \ell)$ is Noetherian.
\end{itemize}
\end{Proposition}

\begin{Proposition}[$\rca$]\label{prop-AlexFiner}
Let $Q$ be a quasi-order. Then $\alex(Q)$ is effectively finer than
$\upper(Q)$.
\end{Proposition}

\begin{proof}
Let $Q$ be a quasi-order, let $((U_q)_{q \in Q}, k)$ be the base for the
Alexandroff topology on $Q$, and let $((V_{\bs i})_{\bs i \in \Pf(Q)}, \ell)$
be the base for the upper topology on $Q$. Define $f \colon \Pf(Q) \times \Nb
\imp \Pf(Q)$ by
\begin{align*}
f(\bs i, n) =
\begin{cases}
\{n\} & \text{if $n \in Q \setminus (\bs{i}\da)$}\\
\emptyset & \text{otherwise}.
\end{cases}
\end{align*}
Then for each $\bs i \in \Pf(Q)$, $\bigcup_{n \in \Nb}\bigcup_{q \in f(\bs i,n)}U_q =
\bigcup_{q \in Q \setminus (\bs{i}\da)} U_q = Q \setminus
(\bs{i}\da) = V_{\bs i}$. So $f$ witnesses that $\alex(Q)$ is effectively
finer that $\upper(Q)$.
\end{proof}

The basic relationships among a quasi-order, its Alexandroff topology, and
its upper topology are provable in $\rca$.

\begin{Proposition}[$\rca$]\label{prop-WQOvsTopInRCA}
Let $Q$ be a quasi-order.
\begin{itemize}
\item[(i)] If $\alex(Q)$ Noetherian, then $\upper(Q)$ Noetherian.

\item[(ii)] $Q$ is a wqo if and only if $\alex(Q)$ is Noetherian.
\end{itemize}
\end{Proposition}

\begin{proof}
Item~(i) follows from Proposition~\ref{prop-AlexFiner} and
Proposition~\ref{prop-finer}.

For item~(ii), first suppose that $\alex(Q)$ is not Noetherian, and let
$(G_n)_{n \in \Nb}$ be an ascending sequence of effectively open sets that
does not stabilize, meaning that $\forall n(G_n\subseteq G_{n+1})$ and $(\forall
n)(\exists m>n)(G_n\subsetneqq G_m)$. We recursively define a bad sequence
$(q_i)_{i \in \Nb}$ of elements of $Q$ together with a sequence of indices
$(n_i)_{i \in \Nb}$ such that $\forall i(q_i \in G_{n_i})$. Suppose that
$(q_i)_{i<k}$ and $(n_i)_{i<k}$ have been defined so that $(q_i)_{i<k}$ is a
finite bad sequence and that $(\forall i < k)(q_i \in G_{n_i})$. Search for a
$q_k$ and $n_k$ such that $q_k \in G_{n_k}$ and $(\forall i < k)(q_i \nleq_Q
q_k)$. Such a pair must exist because there is an $m$ such that $\bigcup_{i <
k}G_{n_i} \subsetneqq G_m$, and in such a $G_m$ there must be a $q$ such that
$(\forall i < k)(q_i \nleq_Q q)$.

Conversely, suppose that $Q$ is not a wqo, and let $(q_i)_{i \in \Nb}$ be a
bad sequence of elements of $Q$. Then the sequence $(G_n)_{n \in \Nb}$, where
$G_n = \{q_i : i < n\}\ua$ for each $n \in \Nb$, is an ascending sequence of
effectively open sets that does not stabilize (in fact, $G_n \subsetneqq
G_{n+1}$ for each $n$).
\end{proof}

Our analysis immediately yields the first two forward directions of Theorem
\ref{thm-BigEquiv}.

\begin{Theorem}[$\aca$]\label{thm-FlatInACA}
If $Q$ is a wqo, then $\alex(\Pf^\flat(Q))$ and $\upper(\Pf^\flat(Q))$ are
Noetherian.
\end{Theorem}
\begin{proof}
Let $Q$ be a wqo. Then $\Pf^\flat(Q)$ is a wqo by
Theorem~\ref{thm-FlatWQOinACA}, so $\alex(\Pf^\flat(Q))$ and
$\upper(\Pf^\flat(Q))$ are Noetherian by
Proposition~\ref{prop-WQOvsTopInRCA}.
\end{proof}

We defer the proof in $\aca$ of the statement ``if $Q$ is a wqo, then
$\upper(\Pf^\sharp(Q))$ is Noetherian'' to Corollary
\ref{thm-SharpInACA}, because a direct proof would essentially repeat the
proof of Theorem~\ref{thm-UnctblSharpInACA}, which is the analogous theorem
for the more general $\Ps^\sharp(Q)$ case.

\subsection{Uncountable second-countable spaces}

If $Q$ is an infinite quasi-order, then $\Ps(Q)$ is uncountable and thus
$\upper(\Ps^\flat(Q))$ and $\upper(\Ps^\sharp(Q))$ cannot be coded as
countable second-countable spaces. However, in second-order arithmetic we can
still express, for example, that the sets $E_0, \dots, E_{n-1}$ code the
basic closed set $\{E_0, \dots, E_{n-1}\}\da^\flat$ by defining $A \in \{E_0,
\dots, E_{n-1}\}\da^\flat$ to mean that $(\exists i < n)(A \leq_Q^\flat
E_i)$. (Notice here that we use the notation `$\da^\flat$' to emphasize that
the downward closure is with respect to $\leq_Q^\flat$. We similarly use
`$\da^\sharp$' to denote the downward closure with respect to
$\leq_Q^\sharp$.) Although not immediately obvious from the definitions, it
is the case that the spaces $\upper(\Ps^\flat(Q))$ and
$\upper(\Ps^\sharp(Q))$ are second-countable (as
Proposition~\ref{prop-FlatSecondCtblACA} and
Proposition~\ref{prop-SharpSecondCtblRCA} below imply).
This situation and Definition~\ref{def-CSCspace} inspire the following
meta-definition, each instance of which is made in $\rca$.

\begin{Definition}[instance-wise in $\rca$]\label{def-GSCspace}
A (general) \emph{second-countable space} is coded by a set $I \subseteq \Nb$
and formulas $\varphi(X)$, $\Psi_=(X,Y)$, and $\Psi_\in(X,n)$ (possibly with
undisplayed parameters) such that the following properties hold.
\begin{itemize}
\item If $\varphi(X)$, then $\Psi_\in(X,i)$ for some $i \in I$.

\item If $\varphi(X)$, $\Psi_\in(X,i)$, and $\Psi_\in(X,j)$ for some $i,j
    \in I$, then there is a $k \in I$ such that $\Psi_\in(X,k)$ and
    $\forall Y[\Psi_\in(Y,k) \imp (\Psi_\in(Y,i) \andd \Psi_\in(Y,j))]$.

\item If $\varphi(X)$, $\varphi(Y)$, $\Psi_\in(X,i)$ for an $i \in I$, and
    $\Psi_=(X,Y)$, then $\Psi_\in(Y,i)$.
\end{itemize}
\end{Definition}

The intuition behind Definition~\ref{def-GSCspace} is that $I$ is a set of
codes for open sets, $\varphi(X)$ says ``$X$ codes a point'', $\Psi_=(X,Y)$
says ``$X$ and $Y$ code the same point'', and $\Psi_\in(X,i)$ says ``the
point coded by $X$ is in the open set coded by $i$''. Effectively open sets
and effectively closed sets are coded as they are in the countable case. A
function $h \colon \Nb \imp \Pf(I)$ codes the effectively open set $G_h =
\bigcup_{n \in \Nb}\bigcup_{i \in h(n)}\{X : \varphi(X) \andd
\Psi_\in(X,i)\}$ and the effectively closed set $F_h = \bigcap_{n \in
\Nb}\bigcap_{i \in h(n)}\{X : \varphi(X) \andd \neg \Psi_\in(X,i)\}$. Again,
`$X \in G_h$' is an abbreviation for the formula `$\varphi(X) \andd (\exists
n \in \Nb)(\exists i \in h(n))\Psi_\in(X,i)$', and similarly for `$X \in
F_h$'. Sequences of effectively open sets and sequences of effectively closed
sets are coded by functions $g \colon \Nb \times \Nb \imp \Pf(I)$, with
$g(n,\cdot)$ coding the $n$\textsuperscript{th} set in the sequence.

As an example, the typical coding of complete separable metric spaces in
$\rca$ (see \cite[Section~II.5]{Simpson:2009vv}) fits nicely into this
framework. Here we first fix a set $A$ and a metric $d \colon A \times A \imp
\Rb$, and we let $I = A \times \Qb^+$. Then we let $\varphi(X)$ be a formula
expressing that $X$ is a rapidly converging Cauchy sequence of points in $A$,
$\Psi_=(X,Y)$ be a formula expressing that the distance between the point
coded by $X$ and the point coded by $Y$ is $0$, and $\Psi_\in(X, \la a, q
\ra)$ be a formula expressing that the distance between $X$ and $a$ is less
than $q$.

Our framework easily accommodates also the countably based MF spaces 
studied in \cite{Mummert06} (by \cite{MumSte} these are exactly the
second-countable $T_1$ spaces with the strong Choquet property), although in
this case the existence of some $X$ satisfying $\varphi(X)$ in general requires
$\aca$, as shown in \cite{LemMum}.

We also define compact spaces and Noetherian spaces as in the countable case.

\begin{Definition}[$\rca$]
A second-countable space coded by $I$, $\varphi$, $\Psi_=$, and $\Psi_\in$ is
$\emph{compact}$ if for every $h \colon \Nb \imp \Pf(I)$ such that $\forall
X(\varphi(X) \imp (\exists n \in \Nb)(\exists i \in h(n))\Psi_\in(X,i))$,
there is an $N \in \Nb$ such that $\forall X(\varphi(X) \imp (\exists n <
N)(\exists i \in h(n))\Psi_\in(X,i))$.

Similarly, an effectively open set $G_h$ in a second-countable space is
\emph{compact} if for every $f \colon \Nb \imp \Pf(I)$ such that $\forall X(X
\in G_h \imp (\exists n \in \Nb)(\exists i \in f(n))\Psi_\in(X,i))$, there is
an $N \in \Nb$ such that $\forall X(X \in G_h \imp (\exists n < N)(\exists i
\in f(n))\Psi_\in(X,i))$.
\end{Definition}

The equivalent characterizations of a Noetherian space given in
Proposition~3.1 (i), (ii), (iv) and (v) are also equivalent in the
uncountable case.  We omit the ``every subspace is compact'' characterization
because quantifying over subspaces of an uncountable space is difficult. One
could quantify over a parameterized collection of subspaces of an uncountable
space via a formula $\theta(X,Y)$ such that $\forall X \forall Y(\theta(X,Y)
\imp \varphi(X))$, in which case each $Y$ corresponds to a subspace, but this
is not useful for our purposes.

\begin{Proposition}[$\rca$]\label{prop-UnctblNoethEquiv}
For a second-countable space, the following statements are equivalent.
\begin{itemize}
\item[(i)] Every effectively open set is compact.

\item[(ii)] For every effectively open set $G_h$, there is an $N \in \Nb$
    such  that $\forall X (X \in G_h \biimp (\exists n < N)(\exists i \in
    h(n))\Psi_\in(X,i))$

\item[(iii)] For every sequence $(G_n)_{n \in \Nb}$ of effectively open
    sets  such that $\forall n(G_n \subseteq G_{n+1})$ there is an $N$ such
    that $(\forall n > N)(G_n = G_N)$.

\item[(iv)] For every sequence $(F_n)_{n \in \Nb}$ of effectively closed
    sets  such that $\forall n(F_n \supseteq F_{n+1})$ there is an $N$ such
    that $(\forall n > N)(F_n = F_N)$.
\end{itemize}
\end{Proposition}

\begin{Definition}[$\rca$]
A second-countable space is \emph{Noetherian} if it satisfies any of the
equivalent conditions from Proposition~\ref{prop-UnctblNoethEquiv}.
\end{Definition}

We now define $\upper(\Ps^\flat(Q))$ and $\upper(\Ps^\sharp(Q))$ as
second-countable spaces.

\begin{Definition}[$\rca$]\label{def-UnctblUpperFlat}
Let $Q$ be a quasi-order. The second-countable space $\upper(\Ps^\flat(Q))$
is coded by the set $I = \Pf(Q)$ and the formulas
\begin{itemize}
\item $\varphi(X) \coloneqq X \subseteq Q$;
\item $\Psi_=(X, Y) \coloneqq X = Y$;
\item $\Psi_\in(X, \bs i) \coloneqq \bs i \subseteq X\da$.
\end{itemize}
\end{Definition}

Notice that $\bs i = \emptyset$ codes the whole space and that the code for
the intersection of the open sets coded by $\bs i$ and $\bs j$ is simply $\bs
i \cup \bs j$. The idea behind $\Psi_\in(X, \bs i)$ is that $\bs i$ codes the
complement of the basic closed set $\{Q \setminus (q\ua) : q \in \bs
i\}\da^\flat$, whence
\begin{align*}
X \notin \{Q \setminus (q\ua) : q \in \bs i\}\da^\flat &
 \Biimp (\forall q \in \bs i)[X \nleq_Q^\flat Q \setminus (q\ua)]\\
 & \Biimp (\forall q \in \bs i)[X \nsubseteq Q \setminus (q\ua)]\\
 &\Biimp (\forall q \in \bs i)[q \in X\da] \Biimp \bs i \subseteq X\da.
\end{align*}

The basic closed sets of the upper topology on $\Ps^\flat(Q)$ are those of
the form $\{E_0, \dots, E_{n-1}\}\da^\flat$ for arbitrary subsets $E_0,
\dots, E_{n-1}$ of $Q$, whereas in Definition~\ref{def-UnctblUpperFlat} we
defined the basic closed sets to be those of the form $\{Q \setminus
(q_0\ua), \dots, Q \setminus (q_{n-1}\ua)\}\da^\flat$ for $q_0, \dots,
q_{n-1} \in Q$. Thus to show that our definition really captures the upper
topology on $\Ps^\flat(Q)$, we need to show that every $\{E_0, \dots,
E_{n-1}\}\da^\flat$ is effectively closed in the topology of
Definition~\ref{def-UnctblUpperFlat}. In fact, it suffices to show that every
set $\{E\}\da^\flat$ is effectively closed in that topology because the
effectively closed sets are closed under finite unions, and $\{E_0, \dots,
E_{n-1}\}\da^\flat = \{E_0\}\da^\flat \cup \cdots \cup \{E_{n-1}\}\da^\flat$.
Unfortunately, as the next proposition shows, proving that $\{E\}\da^\flat$
is effectively closed in the topology of Definition~\ref{def-UnctblUpperFlat}
requires $\aca$ in general, even when $Q$ is a well-order.

\begin{Proposition}\label{prop-FlatSecondCtblACA}
The following are equivalent over $\rca$.
\begin{itemize}
\item[(i)] $\aca$.
\item[(ii)] If $Q$ is a quasi-order and $E \subseteq Q$, then
    $\{E\}\da^\flat$ is effectively closed in $\upper(\Ps^\flat(Q))$.
\item[(iii)] If $W$ is a well-order and $E \subseteq W$, then
    $\{E\}\da^\flat$ is effectively closed in $\upper(\Ps^\flat(W))$.
\end{itemize}
\end{Proposition}

\begin{proof}
For (i)$\Imp$(ii), let $Q$ be a quasi-order and let $E \subseteq Q$. Using
$\aca$ to obtain the set $E\da$, we can define a code for the effectively
closed set $F = \bigcap_{q \notin E\da}\{Q \setminus (q\ua)\}\da^\flat$.  Then for
any $X \subseteq Q$,
\begin{align*}
X\in \{E\}\da^\flat \ \Biimp \ X \subseteq E\da \ \Biimp \
(\forall q \notin E\da)[X \subseteq Q \setminus (q\ua)] \ \Biimp \ X \in F.
\end{align*}
Thus $F = \{E\}\da^\flat$, and so $\{E\}\da^\flat$ is effectively closed.

The implication (ii)$\Imp$(iii) is clear.

For (iii)$\Imp$(i), let $f \colon \Nb \imp \Nb$ be an injection. By
Lemma~\ref{lem-ACAinjection}, it suffices to show that the range of $f$
exists. Let $W$ be a linear order with the following properties:
\begin{itemize}
\item $W$ has order-type $\omega + \omega^*$ (that is, every element of $W$
    has either finitely many predecessors or finitely many successors, and
    there are infinitely many instances of each);
\item if $W$ is not a well-order, then the range of $f$ exists; and
\item the $\omega$ part of $W$ is $\Sigma^0_1$ in $f$.
\end{itemize}
That such a $W$ can be constructed in $\rca$ is well-known (see for example
\cite[Lemma~4.2]{MarSho11}). In fact, our main reversals in the next section
are based on the construction of generalizations of such a $W$, and one may
take $W = \Xi_f(\{x\},x)$, where $\Xi_f(\{x\},x)$ is the partial order (in
this case linear order) from Definition~\ref{def-Xi}.

If $W$ is not a well-order, then the range of $f$ exists by the assumptions
on $W$. So suppose that $W$ is a well-order. Let $E$ be an infinite subset of
the $\omega$ part of $W$, which exists for the same reason that the $X'$ in
the proof of Proposition~\ref{prop-NoethEquiv} exists because the $\omega$
part of $W$ is infinite and $\Sigma^0_1$ in $f$. By (iii), $\{E\}\da^\flat$
is effectively closed, so there is an $h \colon \Nb \imp \Pf(\Pf(Q))$ such
that $\{E\}\da^\flat=F_h$. For each $w \in W$,
$\{w\} \in F_h = \{E\}\da^\flat$ if and only if $w \in E\da$ if and only if
$w$ is in the $\omega$ part of $W$. On the other hand, by definition $\{w\}
\in F_h$ if and only if $\forall k (\forall \bs i \in h(k))(\bs i \nsubseteq
w\da)$, which is $\Pi^0_1$. Thus we have a $\Pi^0_1$ definition of the
$\omega$ part of $W$. Thus by $\Delta^0_1$ comprehension, the $\omega$ part
of $W$ exists. Therefore the $\omega^*$ part of $W$ also exists,
contradicting that $W$ is a well-order.
\end{proof}

The previous proposition is not merely an artifact of a poorly chosen base in
the definition of $\upper(\Ps^\flat(Q))$. In fact, the proof of the reversal
goes through whenever $\Psi_\in$ is defined in such a way that
$\Psi_\in(\{b\}, i)$ is $\Sigma^0_1$. Thus we can see Proposition
\ref{prop-FlatSecondCtblACA} as expressing that the second-countability of
the upper topology on $\Ps^\flat(Q)$ is equivalent to $\aca$.

\begin{Definition}[$\rca$]\label{def-UnctblUpperSharp}
Let $Q$ be a quasi-order. The second-countable space $\upper(\Ps^\sharp(Q))$
is coded by the set $I = \Pf(Q)$ and the formulas
\begin{itemize}
\item $\varphi(X) \coloneqq X \subseteq Q$;
\item $\Psi_=(X, Y) \coloneqq X = Y$;
\item $\Psi_\in(X, \bs i) \coloneqq \bs i \cap X\ua = \emptyset$.
\end{itemize}
\end{Definition}

Again, $\bs i = \emptyset$ codes the whole space, and the code for the
intersection of the open sets coded by $\bs i$ and $\bs j$ is $\bs i \cup \bs
j$. The idea behind $\Psi_\in(X, \bs i)$ is that $\bs i$ codes the complement
of the basic closed set $\{\{q\} : q \in \bs i\}\da^\sharp$, whence
\begin{align*}
X \notin \{\{q\} : q \in \bs i\}\da^\sharp &
\Biimp (\forall q \in \bs i)(X \nleq_Q^\sharp \{q\}) \\ &
\Biimp (\forall q \in \bs i)(q \notin X\ua)
\Biimp \bs i \cap X\ua = \emptyset.
\end{align*}

Unlike in the $\flat$ case, $\rca$ suffices to show that
$\upper(\Ps^\sharp(Q))$ as defined in Definition~\ref{def-UnctblUpperSharp}
indeed captures the upper topology on $\Ps^\sharp(Q)$.

\begin{Proposition}[$\rca$]\label{prop-SharpSecondCtblRCA}
If $Q$ is a quasi-order and $E \subseteq Q$, then $\{E\}\da^\sharp$ is
effectively closed in $\upper(\Ps^\sharp(Q))$.
\end{Proposition}

\begin{proof}
Let $Q$ be a quasi-order, and let $E \subseteq Q$. Then $\{E\}\da^\sharp $ is
exactly the effectively closed set $\bigcap_{e \in E}\{e\}\da^\sharp$ because
\begin{align*}
X\in\{E\}\da^\sharp \ \Biimp \  E \subseteq X\ua  \ \Biimp \  (\forall e \in E)(e \in X
\ua)  \ \Biimp \  X \in \bigcap_{e \in E}\{e\}\da^\sharp
\end{align*}
for any $X \subseteq Q$.
\end{proof}

We now examine the correspondences between the countable spaces and the
uncountable spaces. Our goal is to prove, in $\rca$, that if $Q$ is a
quasi-order, then $\upper(\Ps^\flat(Q))$ is Noetherian implies that
$\upper(\Pf^\flat(Q))$ is Noetherian and likewise with `$\sharp$' in place of
`$\flat$.'

We warn the reader that the upper topology on $\Pf^\flat(Q)$ is in general
not the same as the subspace topology on $\Pf(Q)$ induced by the upper
topology on $\Ps^\flat(Q)$. For example, if $Q$ is an infinite antichain and
$q \in Q$, then the basic closed set $\{Q \setminus \{q\}\}\da^\flat$ in the
upper topology on $\Ps^\flat(Q)$ induces the closed set $\{\bs x \in \Pf(Q) :
q \notin \bs x\}$ in the subspace topology on $\Pf(Q)$, but this set is not
closed in the upper topology on $\Pf^\flat(Q)$. To see this, observe that a
closed set in the upper topology on $\Pf^\flat(Q)$ that is not the whole space
must be contained in a basic closed set of the form $\{\bs{e}_i : i <
n\}\da^\flat$, and if $\bs x \in \{\bs{e}_i : i < n\}\da^\flat$, then $|\bs
x| \leq \max\{|\bs{e}_i| : i < n\}$. A similar argument shows that these two
topologies need not be the same even if $Q$ is a well-order. Let $Q =
\omega+1$. Then the closed set $\omega\da^\flat$ in $\Ps^\flat(Q)$ induces
the closed set $\{\bs x \in \Pf(Q) : \bs x \subseteq \omega\}$ in the
subspace topology on $\Pf(Q)$, but the closed sets in the upper topology on
$\Pf^\flat(Q)$ are all of the form $\{\bs x \in \Pf(Q) : \bs x \subseteq
q\da\}$ for some $q \in Q$.

However, the upper topology on $\Pf^\sharp(Q)$ is indeed the same as the
subspace topology on $\Pf(Q)$ induced by the upper topology on
$\Ps^\sharp(Q)$. This is easy to see because $\{\{q_i\} : i < n\}\da^\sharp$
contains the same finite sets regardless of whether it is interpreted as a
basic closed set in the upper topology on $\Ps^\sharp(Q)$ or as a basic
closed set in the upper topology on $\Pf^\sharp(Q)$.

\begin{Lemma}[$\rca$]\label{lem-CtblUnctblTranslate}
Let $Q$ be a quasi-order.
\begin{itemize}
\item[(i)] For every effectively closed set $F$ in $\upper(\Pf^\flat(Q))$,
    there is an effectively closed set $\MP F$ in $\upper(\Ps^\flat(Q))$
    such that $(\forall \bs x \in \Pf(Q))(\bs x \in \MP F \biimp \bs x \in
    F)$.

\item[(ii)] For every effectively closed set $F$ in
    $\upper(\Pf^\sharp(Q))$,  there is an effectively closed set $\MP F$ in
    $\upper(\Ps^\sharp(Q))$ such that $(\forall \bs x \in \Pf(Q))(\bs x \in
    \MP F \biimp \bs x \in F)$.
\end{itemize}
\end{Lemma}

\begin{proof}
We first prove (i) for basic closed sets. A basic closed set in
$\upper(\Pf^\flat(Q))$ has the form $E\da^\flat$ for some $E \in
\Pf(\Pf(Q))$. Suppose that $E = \{\bs{e}_0,\dots,\bs{e}_{n-1}\}$, and
consider the effectively closed set $\MP{F}_E$ in $\upper(\Ps^\flat(Q))$
given by
\begin{align*}
\MP{F}_E = \bigcap_{\substack{(q_0,\dots,q_{n-1}) \in Q^n \\ (\forall i < n)(q_i 
\notin
\bs{e}_i\da)}}\{Q \setminus (q_i\ua) : i < n\}\da^\flat.
\end{align*}
We show that $(\forall \bs x \in \Pf(Q))(\bs x \in \MP{F}_E \biimp \bs x \in
E\da^\flat)$. Suppose that $\bs x \in E\da^\flat$. Then there is an $i<n$
such that $\bs x \leq_Q^\flat \bs{e}_i$, so $\bs x \subseteq \bs{e}_i\da$,
and therefore $(\forall q \notin \bs{e}_i\da)[\bs x \subseteq Q \setminus
(q\ua)]$. Hence $\bs x \in \MP{F}_E$. Conversely, suppose that $x \notin
E\da^\flat$. Then $(\forall i < n)(\bs x \nleq_Q^\flat \bs{e}_i)$, so
$(\forall i < n)(\bs x \nsubseteq \bs{e}_i\da)$, and finally $(\forall i <
n)(\exists q_i \in \bs x)(q_i \notin \bs{e}_i\da)$. Then $\MP{F}_E \subseteq
\{Q \setminus (q_i\ua) : i < n\}\da^\flat$ and $\bs x \notin \{Q \setminus
(q_i\ua) : i < n\}\da^\flat$. Thus $\bs x \notin \MP{F}_E$.

To complete the proof of (i), let us now consider the effectively closed set
$F_h = \bigcap_{n \in \Nb}\bigcap_{E \in h(n)}E\da^\flat$ in
$\upper(\Pf^\flat(Q))$ coded by $h \colon \Nb \imp \Pf(\Pf(\Pf(Q)))$.
The procedure that produces (the code for) $\MP{F}_E$ given $E \in
\Pf(\Pf(Q))$ is uniform in $E$, so from $h$ we can produce $g \colon \Nb
\times \Nb \imp \Pf(\Pf(Q))$ such that, for every $n \in \Nb$,
$\MP{F}_{g(n,\cdot)} = \bigcap_{E \in h(n)}\MP{F}_E$. The intersection of a
sequence of effectively closed sets is also an effectively closed set, so
from $g$ we can produce a code for the effectively closed set $\MP F =
\bigcap_{n \in \Nb}\MP{F}_{g(n,\cdot)} = \bigcap_{n \in \Nb}\bigcap_{E \in
h(n)}\MP{F}_E$. Then, for any $\bs x \in \Pf(Q)$,
\begin{align*}
\bs x \in \MP F \Biimp \bs x \in \bigcap_{n \in \Nb}\bigcap_{E \in h(n)}\MP{F}_E
\Biimp \bs x \in \bigcap_{n \in \Nb}\bigcap_{E \in h(n)}E\da^\flat \Biimp \bs x \in
F_h.
\end{align*}

Now we prove (ii) for basic closed sets. A basic closed set in
$\upper(\Pf^\sharp(Q))$ has the form $E\da^\sharp$ for some $E \in
\Pf(\Pf(Q))$. Suppose that $E = \{\bs{e}_0,\dots,\bs{e}_{n-1}\}$, and
consider the effectively closed set $\MP{F}_E$ in $\upper(\Ps^\sharp(Q))$
given by
\begin{align*}
\MP{F}_E = \bigcap_{(q_0,\dots,q_{n-1}) \in \bs{e}_0 \times \cdots \times \bs{e}
_{n-1}}\{\{q_0\}, \dots, \{q_{n-1}\}\}\da^\sharp.
\end{align*}
We show that $(\forall \bs x \in \Pf(Q))(\bs x \in \MP{F}_E \biimp \bs x \in
E\da^\sharp)$. Suppose that $\bs x \in E\da^\sharp$. Then there is an $i<n$
such that $\bs x\leq_Q^\sharp \bs{e}_i$, so $\bs{e}_i \subseteq \bs{x}\ua$,
and therefore $(\forall q \in \bs{e}_i)(q \in \bs{x}\ua)$.
Hence $\bs x \in \MP{F}_E$. Conversely, suppose that $x \notin E\da^\sharp$.
Then $(\forall i < n)(\bs x \nleq_Q^\sharp \bs{e}_i)$, so $(\forall i <
n)(\bs{e}_i \nsubseteq \bs{x}\ua)$, and therefore $(\forall i < n)(\exists q_i \in
\bs{e}_i)(q_i \notin \bs{x}\ua)$. Then $\MP{F}_E \subseteq \{\{q_0\}, \dots,
\{q_{n-1}\}\}\da^\sharp$ and $\bs x \notin
\{\{q_0\}, \dots, \{q_{n-1}\}\}\da^\sharp$. Thus $\bs x \notin \MP{F}_E$.

To complete the proof of (ii), given an effectively closed set $F$ in
$\upper(\Pf^\sharp(Q))$, we can produce an effectively closed set $\MP F$ in
$\upper(\Ps^\sharp(Q))$ such that $(\forall \bs x \in \Pf(Q))(\bs x \in \MP F
\biimp \bs x \in F)$ just as in the proof of (i).
\end{proof}

\begin{Theorem}[$\rca$]\label{thm-UnctblNoethImpliesCtblNoeth}
Let $Q$ be a quasi-order.
\begin{itemize}
\item[(i)] If $\upper(\Ps^\flat(Q))$ is Noetherian, then
    $\upper(\Pf^\flat(Q))$ is Noetherian.

\item[(ii)] If $\upper(\Ps^\sharp(Q))$ is Noetherian, then
    $\upper(\Pf^\sharp(Q))$ is Noetherian.
\end{itemize}
\end{Theorem}

\begin{proof}
For (i), suppose that $\upper(\Pf^\flat(Q))$ is not Noetherian, and let
$(F_n)_{n \in \Nb}$ be a non-stabilizing descending sequence of effectively
closed sets in $\upper(\Pf^\flat(Q))$. The proof of
Lemma~\ref{lem-CtblUnctblTranslate}~(i) is uniform, so from $(F_n)_{n \in
\Nb}$ we can produce a sequence $(\MP{F}_n)_{n \in \Nb}$ of effectively
closed sets in $\upper(\Ps^\flat(Q))$ such that $(\forall n \in \Nb)(\forall
\bs x \in \Pf(Q))(\bs x \in \MP{F}_n \biimp \bs x \in F_n)$. Define a new
sequence $(\MP{H}_n)_{n \in \Nb}$ by $\MP{H}_n = \bigcap_{m \leq n}\MP{F}_m$
for each $n \in \Nb$. Then $(\MP{H}_n)_{n \in \Nb}$ is a descending sequence of
closed sets in $\upper(\Ps^\flat(Q))$ that does not stabilize because
$(F_n)_{n \in \Nb}$ does not stabilize and $(\forall n \in \Nb)(\forall \bs x
\in \Pf(Q))(\bs x \in \MP{H}_n \biimp \bs x \in F_n)$. Hence
$\upper(\Ps^\flat(Q))$ is not Noetherian.

The proof of (ii) is the same, except we use
Lemma~\ref{lem-CtblUnctblTranslate}~(ii) in place of
Lemma~\ref{lem-CtblUnctblTranslate}~(i).
\end{proof}

Theorem~\ref{thm-UnctblNoethImpliesCtblNoeth} tells us that in the forward
direction we need only work with the uncountable spaces and that in the
reverse direction we need only work with the countable spaces.

\begin{Theorem}[$\aca$]\label{thm-UnctblFlatInACA}
If $Q$ is a wqo, then $\upper(\Ps^\flat(Q))$ is Noetherian.
\end{Theorem}

\begin{proof}
We prove the contrapositive. Let $Q$ be a quasi-order, suppose that
$\upper(\Ps^\flat(Q))$ is not Noetherian, and let $(F_n)_{n \in \Nb}$ be a
non-stabilizing descending sequence of effectively closed sets. Our goal is
to build a bad sequence in $\Pf^\flat(Q)$, thereby proving that
$\Pf^\flat(Q)$ is not a wqo and hence, by Theorem~\ref{thm-FlatWQOinACA},
that $Q$ is not a wqo.

\begin{Claim}
If $F$ is an effectively closed set in $\upper(\Ps^\flat(Q))$
and $A \subseteq Q$, then $A \in F$ if and only if $\Pf(A) \subseteq F$.
\end{Claim}

\begin{proof}[Proof of claim]
The forward direction is clear because effectively closed sets are closed
downward under $\leq_Q^\flat$, and $B \leq_Q^\flat A$ whenever $B \subseteq A
$.  For the reverse direction, suppose that $F$ is coded by $h \colon \Nb \imp
\Pf(\Pf(Q))$. Then $A \notin F$ means that $(\exists n \in \Nb)(\exists \bs i
\in h(n))(\bs i \subseteq A\da)$. As the witnessing $\bs i$ is finite, there
is a finite $\bs a \subseteq A$ such that $\bs i \subseteq \bs{a}\da$, and
this $\bs a$ satisfies $\bs a \notin F$.
\end{proof}

It follows from the claim that if $F_n \setminus F_{n+1} \neq \emptyset$ for
some $n \in \Nb$, then there is a finite $\bs{a} \in F_n \setminus F_{n+1}$.
Suppose we have constructed a sequence $(\bs{a}_i)_{i<n}$ of elements of
$\Pf(Q)$ along with an increasing sequence $(m_i)_{i<n}$ such that $(\forall
i < n)(\bs{a}_i \in F_{m_i} \setminus F_{m_i + 1})$. As $(F_n)_{n \in \Nb}$
is non-stabilizing, we may extend the sequence by finding an $m_n > m_{n-1}$
(or an $m_n \geq 0$ if $n=0$) and an $\bs{a}_n \in \Pf(Q)$ that is in
$F_{m_n} \setminus F_{m_n + 1}$. In the end, $ (\bs{a}_n)_{n \in \Nb}$ is a
bad sequence because, for each $n \in \Nb$, $\bs{a}_n \in F_{m_n}$ but
$(\forall i < n)(\bs{a}_i \notin F_{m_n})$, which means that $ (\forall i <
n)(\bs{a}_i \nleq_Q^\flat \bs{a}_n)$.
\end{proof}

\begin{Theorem}[$\aca$]\label{thm-UnctblSharpInACA}
If $Q$ is a wqo, then $\upper(\Ps^\sharp(Q))$ is Noetherian.
\end{Theorem}

\begin{proof}
Let $Q$ be a wqo. Suppose for a contradiction that $\upper(\Ps^\sharp(Q))$ is
not Noetherian, and let $(F_n)_{n \in \Nb}$ be a non-stabilizing descending
sequence of effectively closed sets. Our goal is to construct a bad sequence
$(q_n)_{n \in \Nb}$ of elements of $Q$, contradicting that $Q$ is a wqo.

\begin{MClaim}\label{claim-finite}
If $F$ is an effectively closed set in $\upper(\Ps^\sharp(Q))$ and $A
\subseteq Q$, then $A \in F $ if and only if $ (\exists \bs a \in \Pf(A))(\bs
a \in F)$
\end{MClaim}

\begin{proof}[Proof of claim]
The backwards direction is clear because effectively closed sets are closed
downward under $\leq_Q^\sharp$, and $A \leq_Q^\sharp B$ whenever $B
\subseteq A$.

For the forward direction, the fact that $Q$ is a wqo implies that there is a
finite $\bs a \subseteq A$ such that $\bs{a} \leq_Q^\sharp A$, for otherwise
it is easy to construct a bad sequence by choosing elements of $A$ (see
\cite[Lemma~4.8]{M2005}).
\end{proof}

It follows from Claim~\ref{claim-finite} that two effectively closed sets are equal if
and only if they agree on $\Pf(Q)$.  Therefore the equality of two effectively closed
sets is an arithmetical property of the sets, and whether or not a descending
sequence of effectively closed sets stabilizes is an arithmetical property of the
sequence.

Suppose we have constructed a finite bad sequence $(q_i)_{i < k}$ of elements
of $Q$ such that the sequence $(F'_n)_{n \in \Nb}$ given by $F'_n = F_n \cap
\bigcap_{i < k}\{q_i\}\da^\sharp$ for each $n \in \Nb$ does not stabilize.
Search for an $\bs a \in \Pf(Q)$ and an $\ell$ such that $\bs a \in F'_\ell
\setminus F'_{\ell+1}$. As $\bs a \in \bigcap_{i < k}\{q_i\}\da^\sharp$, it
must be that $\bs a \notin F_{\ell+1}$ and hence that $\bs a \notin \{\{r_j\}
: j < m\}\da^\sharp$ for some superset $\{\{r_j\} : j < m\}\da^\sharp$ of
$F_{\ell+1}$. Notice that $(\forall i < k)(\forall j < m)(q_i \nleq_Q r_j)$
because if $q_i \leq_Q r_j$ for some $i < k$ and $j < m$, then $\bs a
\leq_Q^\sharp \{q_i\} \leq_Q^\sharp \{r_j\}$ would contradict $\bs a \notin
\{\{r_j\} : j < m\}\da^\sharp$. Thus we could chose any $r_j$ for $j < m$ to
extend our bad sequence. We need to show that at least one such choice allows
us to continue the construction.

\begin{MClaim}\label{claim-stab}
There is $j < m$ such that the sequence $(F'_n \cap \{r_j\}\da^\sharp)_{n
\in \Nb}$ does not stabilize.
\end{MClaim}

\begin{proof}[Proof of claim]
Suppose for a contradiction that the sequence $(F'_n \cap
\{r_j\}\da^\sharp)_{n \in\Nb}$ stabilizes for each $j < m$. Let $N > \ell+1$
be large enough so that $(\forall j < m)(\forall n > N)(F'_n \cap
\{r_j\}\da^\sharp = F'_N \cap \{r_j\}\da^\sharp)$. Such an $N$ exists because
the stabilization of $(F'_n \cap \{r_j\}\da^\sharp)_{n \in\Nb}$ is an
arithmetical property, and $\aca$ proves the bounding axiom for every
arithmetical formula. For all $n \geq N$, we have that

\begin{align*}
\bigcup_{j < m}(F'_n \cap \{r_j\}\da^\sharp) = F'_n \cap \bigcup_{j < m}\{r_j\}\da^
\sharp = F'_n \cap \{\{r_j\} : j < m\}\da^\sharp = F'_n,
\end{align*}
where the last equality holds because $F'_n \subseteq F'_{\ell + 1} \subseteq
\{\{r_j\} : j < m\}\da^\sharp$, and that
\begin{align*}
\bigcup_{j < m}(F'_n \cap \{r_j\}\da^\sharp) = \bigcup_{j < m}(F'_N \cap \{r_j\}\da^
\sharp) = F'_N.
\end{align*}

Thus $(\forall n > N)(F'_n = F'_N)$, contradicting that the sequence
$(F'_n)_{n \in\Nb}$ does not stabilize.
\end{proof}
Let $q_k$ be $r_j$ for the $r_j$ guaranteed by Claim~\ref{claim-stab}. Again,
the procedure for computing $q_k$ is arithmetical because the stabilization
of a sequence is an arithmetical property. Then $ (q_i)_{i < k+1}$ is a bad
sequence and the sequence $(F_n \cap \bigcap_{i < k+1}\{q_i\}\da^\sharp)_{n
\in \Nb}$ does not stabilize, so we may continue the construction and build a
contradictory infinite bad sequence.
\end{proof}

\begin{Corollary}[$\aca$]\label{thm-SharpInACA}
If $Q$ is a wqo, then the countable second-countable space
$\upper(\Pf^\sharp(Q))$ is Noetherian.
\end{Corollary}
\begin{proof}
Immediate from Theorem~\ref{thm-UnctblNoethImpliesCtblNoeth}
and Theorem~\ref{thm-UnctblSharpInACA}.
\end{proof}

A similar corollary can be obtained from
Theorem~\ref{thm-UnctblNoethImpliesCtblNoeth} and
Theorem~\ref{thm-UnctblFlatInACA},
providing a new proof that $\aca$ proves that if $Q$ is a wqo, then the
countable second-countable space $\upper(\Pf^\flat(Q))$ is Noetherian (which
we already saw in Theorem~\ref{thm-FlatInACA}).

Notice also that one could omit the application of
Theorem~\ref{thm-UnctblNoethImpliesCtblNoeth} and prove directly, in $\aca$,
that if $Q$ is a wqo, then $\upper(\Pf^\flat(Q))$ (respectively
$\upper(\Pf^\sharp(Q))$) is Noetherian by implementing the proof of
Theorem~\ref{thm-UnctblFlatInACA} (respectively
Theorem~\ref{thm-UnctblSharpInACA}) in the countable second-countable spaces
setting. It is also possible to give a direct proof of
Theorem~\ref{thm-UnctblFlatInACA} in which one
builds a bad sequence in $Q$ instead of in $\Pf^\flat(Q)$ in the style of the proof 
of
Theorem~\ref{thm-UnctblSharpInACA}. Finally, recall that
Proposition~\ref{prop-FlatSecondCtblACA} shows, essentially, that without
$\aca$ the definition of $\upper(\Ps^\flat(Q))$ as a second-countable space
(Definition~\ref{def-UnctblUpperFlat}) codes a coarser topology than the
upper topology on $\Ps^\flat(Q)$. Nevertheless, we may still give an \emph{ad
hoc} definition of the upper topology on $\Ps^\flat(Q)$ in $\rca$ by
interpreting a sequence $((E^n_i)_{i < m_n})_{n \in \Nb}$ of finite sequences
of subsets of $Q$ as a code for the closed set $\bigcap_{n \in \Nb}\{E^n_i :
i < m_n\}\da^\flat$. Then, by a proof in the style of that of
Theorem~\ref{thm-UnctblSharpInACA}, $\aca$ proves that if $Q$ is a wqo, then
this topology is Noetherian.

\section{The reversals}\label{sec-reversals}

The strategy for reversing, for example, the statement ``if $Q$ is a wqo,
then $\upper(\Pf^\flat(Q))$ is Noetherian'' to $\aca$ is to produce a
recursive quasi-order $Q$ such that $\upper(\Pf^\flat(Q))$ is not Noetherian
as witnessed by some uniformly r.e.\ descending sequence of closed sets, yet
every bad sequence from $Q$ computes $0'$.
In \cites{MarSho11,FriMar12,FriMar14} the main reversals to $\aca$ are based
on the construction of a recursive linear order of type $\omega+\omega^*$ with 
the 
property that every descending sequence computes $0'$ (we used this technique 
in the proof of Proposition \ref{prop-FlatSecondCtblACA}). We generalize this 
construction to partial orders.  Given a finite partial order $P$ and an $x \in P$, we 
define a recursive partial order $Q = \Xi(P,x)$ with the property that every bad
sequence from $Q$ computes $0'$. The special case $P = \{x\}$ produces a
recursive linear order $\Xi(\{x\},x)$ of type $\omega+\omega^*$ in which
every descending sequence computes $0'$. As in the reversals using linear
orders of type $\omega+\omega^*$, the notion of \emph{true stage} is crucial.

\begin{Definition}[$\rca$]\label{def-true}
Let $f \colon \Nb \imp \Nb$ be an injection. An $n \in \Nb$ is
$f$-\emph{true} (or simply \emph{true}) if $(\forall k > n)(f(n) < f(k))$.
An $n \in \Nb$ is $f$-\emph{true} (or simply \emph{true}) \emph{at stage} $s
\in \Nb$ if $n < s$ and $\forall k(n < k \leq s \imp f(n) < f(k))$.
\end{Definition}
The notion of true stages is not new. Dekker \cite{Dek54} introduced this
notion (but he used the term `minimal') to show that every non-recursive
r.e.\ degree contains a hypersimple set.  Indeed, given a recursive
enumeration of a non-recursive r.e.\ set $A$, the set of non-true stages is
hypersimple and Turing equivalent to $A$ (see also
\cite[Theorem~XVI]{Rog87}).  An early use of true stages in reverse
mathematics is in \cite[Section 1]{Shore93}.  In recursion theory, true
stages are also known as non-deficiency stages (see~\cite{Soa87}).

The import of this definition is that the range of an injection $f \colon \Nb
\imp \Nb$ is $\Delta^0_1$ in the join of $f$ and any infinite set $T$ of
$f$-true stages: indeed for any $n \in \Nb$, $\exists m (f(m) = n)$ if and
only if $(\forall m \in T)(f(m) > n \imp (\exists k < m)(f(k) = n))$. Thus
$\rca$ proves that, for any injection $f$, if there is an infinite set of
$f$-true stages, then the range of $f$ exists.

For the purposes of the following definition, given an injection $f\colon\Nb\to\Nb$,
set
\begin{align*}
T_s = \{n < s : \text{$n$ is $f$-true at stage $s$}\},
\end{align*}
and note that $\rca$ proves that the sequence $(T_s)_{s \in \Nb}$ exists.

\begin{Definition}[$\rca$]\label{def-Xi}
Let $f \colon \Nb \imp \Nb$ be an injection, let $P$ be a finite partial
order, and let $x \in P$. We define the partial order
$Q = \Xi_f(P,x)$ as follows.  Make $\Nb$ disjoint copies of $P$ by letting
$P_n = \{n\} \times P$ for each $n \in \Nb$, and let $x_n = (n,x)$ denote the
copy of $x$ in $P_n$.  The domain of $Q$ is $\bigcup_{n \in \Nb}P_n$.  Define $
\leq_Q$ in stages, where at stage $s$, $\leq_Q$ is defined on
$\bigcup_{n \leq s}P_n$.

\begin{itemize}
\item At stage $0$, $\leq_Q$ is simply $\leq_{P_0}$ on $P_0$.
\item Suppose $\leq_Q$ is defined on $\bigcup_{n \leq s} P_s$. There are
    two cases.
\begin{enumerate}[(i)]
\item If $T_{s+1} \subsetneqq T_s \cup \{s\}$, let $n_0$ be the least
    element of $(T_s \cup \{s\}) \setminus T_{s+1}$, and place $P_{s+1}$
    immediately above $x_{n_0}$. That is, place the elements of $P_{s+1}$
    above all $y \in \bigcup_{n \leq s} P_s$ such that $y \leq_Q
    x_{n_0}$, below all $y \in \bigcup_{n \leq s} P_s$ such that $y >_Q
    x_{n_0}$, and incomparable with all $y \in \bigcup_{n \leq s} P_s$
    that are incomparable with $x_{n_0}$.

\item If $T_{s+1} = T_s \cup \{s\}$, place $P_{s+1}$ immediately below
    $x_s$. That is, place the elements of $P_{s+1}$ above all $y \in
    \bigcup_{n \leq s} P_s$ such that $y <_Q x_s$, below all $y \in
    \bigcup_{n \leq s} P_s$ such that $y \geq_Q x_s$, and incomparable
    with all $y \in \bigcup_{n \leq s} P_s$ that are incomparable with
    $x_s$.
\end{enumerate}
In both cases, define $\leq_Q$ to be $\leq_{P_{s+1}}$ on $P_{s+1}$.
\end{itemize}
\end{Definition}

We could extend the construction of Definition~\ref{def-Xi} by starting from
any sequence $(P_n)_{n \in \Nb}$ of finite (or even infinite) quasi-orders
and any choice of elements $x_n \in P_n$ for each $n$, but we have no need
for such generality. We just note that if each $P_n$ is allowed to be
infinite, then Lemma~\ref{lem-Qproperties} is still provable in $\rca$, but
Lemma~\ref{lem-ACAreversal} holds only if each $P_n$ is a wqo, and its proof
requires the infinite pigeonhole principle for an arbitrary number of colors
(i.e., $\forall k \rt^1_k $, which is equivalent to $\mathsf{B} \Sigma^0_2$
over $\rca$~\cite{HirstThesis}).

For the purposes of the next lemmas, $P_m \leq_Q x_n$ means $(\forall z \in
P_m)(z \leq_Q x_n)$, $x_n \leq_Q P_m$ means $(\forall z \in P_m)(x_n \leq_Q
z)$, and $P_m \mid_Q y$ means $(\forall z \in P_m)(z \mid_Q y)$.

\begin{Lemma}[$\rca$]\label{lem-Qproperties}
Let $f \colon \Nb \imp \Nb$, $P$ be a finite partial order, $x \in P$, and
$Q=\Xi_f (P,x)$, and consider $m, n \in \Nb$ with $n < m$.
\begin{enumerate}[(i)]
\item If $n \in T_m$, then $P_m \leq_Q x_n$ and $(\forall y \in P_n)(x_n
    \mid_Q y \imp P_m \mid_Q y)$.
\item If $n \notin T_m$, then $x_n \leq_Q P_m$.
\end{enumerate}
\end{Lemma}

\begin{proof}
We simultaneously prove (i) and (ii) by $\Sigma^0_0$ induction on $m$. The
case $m=0$ is vacuously true.

Consider $m+1$.  First suppose that $T_{m+1} \subsetneqq T_m \cup \{m\}$ and
thus that $P_{m+1}$ is placed immediately above $x_{n_0}$, where $n_0$ is the
least element of $(T_m \cup \{m\}) \setminus T_{m+1}$.  Now, either $n_0=m$ or 
$n_0 \in T_m \setminus T_{m+1}$, and in both cases it must be that $f(m +1) < 
f(n_0)$ and $(\forall k \in (n_0,m])(f(n_0) < f(k))$.  Notice that the interval $(n_0,m]$ 
is empty when $n_0=m$.

For item~(i), suppose that $n<m+1$ is such that $n \in T_{m+1}$. First we
claim that $n < n_0$. As $n \in T_{m+1}$ and $n_0 \notin T_{m+1}$ we have $n
\neq n_0$. Now, if $n_0 < n$, then either $f(n_0) < f(n)$, in which case
$f(m+1) < f(n_0) < f(n)$, contradicting $n \in T_{m+1}$, or $f(n) < f(n_0)$,
contradicting that $n_0 \notin T_{m+1}$ is only witnessed by $m+1 \neq n$.
Hence $n < n_0$ as claimed. This implies that $n \in T_{n_0}$ because $n \in
T_{m+1}$ and $n_0 < m+1$. By the induction hypothesis, $P_{n_0} \leq_Q x_n$
and $(\forall y \in P_n)(x_n \mid_Q y \imp P_{n_0} \mid_Q y)$. Thus $x_{n_0}
\leq_Q x_n$, so $P_{m+1} \leq_Q x_n$ because $P_{m+1}$ is placed 
immediately
above $x_{n_0}$. Furthermore, every $y \in P_n$ that is incomparable with
$x_n$ is incomparable with $x_{n_0}$ and is hence incomparable with every
element of $P_{m+1}$.

For item~(ii), suppose that $n<m+1$ is such that $n \notin T_{m+1}$. If $n =
n_0$, then $P_{m+1}$ is placed immediately above $x_{n_0} = x_n$, as desired.
Suppose $n_0 < n$. Then $n_0 \in T_n$ because $n < m+1$. By the induction
hypothesis, $P_n \leq_Q x_{n_0}$, so $x_n \leq_Q x_{n_0}$. $P_{m+1}$ is
placed immediately above $x_{n_0}$, so $x_n \leq P_{m+1}$. If instead $n <
n_0$, we claim that $n \notin T_{n_0}$. This is clear if $f(n_0) < f(n)$, so
suppose that $f(n) < f(n_0)$. As $n \notin T_{m+1}$, there is a least $k \in
(n,m+1]$ such that $f(k) < f(n)$. If $k = m+1$, then $n \in (T_m \cup \{m\})
\setminus T_{m+1}$, contradicting that $n_0$ was the least such number. If $k
\in (n_0, m]$, then $f(k) < f(n) < f(n_0)$, contradicting that only $m+1$
witnesses that $n_0 \notin T_{m+1}$. Thus $k \in (n, n_0]$, which means that
$k$ witnesses that $n \notin T_{n_0}$, establishing the claim. By the
induction hypothesis, $x_n \leq_Q P_{n_0}$, so $x_n \leq_Q x_{n_0}$.
$P_{m+1}$ is placed immediately above $x_{n_0}$, so $x_n \leq_Q P_{m+1}$.
This concludes the proof of (i)~and~(ii) for $m+1$ in the $T_{m+1}
\subsetneqq T_m \cup \{m\}$ case.

Now suppose that $T_{m+1} = T_m \cup \{m\}$, so that $P_{m+1}$ is placed
immediately below $x_m$. For item~(i), suppose that $n \in T_{m+1}$. If $n =
m$, then $P_{m+1} \leq_Q x_n = x_m$, and every $y \in P_n = P_m$ that is
incomparable with $x_n = x_m$ is incomparable with every element of
$P_{m+1}$. If $n < m$, then $n \in T_m$, so by the induction hypothesis $P_m
\leq_Q x_n$ and $(\forall y \in P_n)(x_n \mid_Q y \imp P_m \mid_Q y)$. Thus
$x_m \leq_Q x_n$, and so $P_{m+1} \leq_Q x_n$ because $P_{m+1}$ is placed
immediately below $x_m$. Furthermore, every $y \in P_n$ such that $x_n \mid_Q
y$ is incomparable with $x_m$ and is hence incomparable with every element of
$P_{m+1}$.

For item~(ii), suppose that $n \notin T_{m+1} = T_m \cup \{m\}$. Then $n < m$
and $n \notin T_m$, so, by the induction hypothesis, $x_n \leq_Q P_m$. Thus
$x_n \leq_Q x_m$. So $x_n \leq_Q P_{m+1}$ because $P_{m+1}$ is placed
immediately below $x_m$. This concludes the proof.
\end{proof}

\begin{Lemma}[$\rca$]\label{lem-ACAreversal}
Let $f \colon \Nb \imp \Nb$, $P$, $x \in P$, and $Q = \Xi_f(P,x)$ be as
above. If $Q$ is not a wqo, then the range of $f$ exists.
\end{Lemma}

\begin{proof}
Suppose that $Q$ is not a wqo, and let $(q_i)_{i \in \Nb}$ be a bad sequence.
We show that $n \in \Nb$ is true if and only if $\exists i (q_i
\leq_Q x_n)$. Thus the set of true stages has both a $\Pi^0_1$ definition (as
in Definition~\ref{def-true}) and a $\Sigma^0_1$ definition, so it exists by
$\Delta^0_1$ comprehension. It follows that the range of $f$ exists as
explained following Definition~\ref{def-true}.

Suppose that $n \in \Nb$ is true. The sequence $(q_i)_{i \in \Nb}$ is
injective and each $P_m$ is of the same finite size, so there must be an $i$
and an $m$ in $\Nb$ with $m>n$ such that $q_i \in P_m$.  As $n$ is a true
stage, $n \in T_m$, so $P_m \leq_Q x_n$ by Lemma~\ref{lem-Qproperties}~(i).
Thus $q_i \leq_Q x_n$ as desired. Conversely, suppose that $n \in \Nb$ is not
true and suppose for a contradiction that $q_i \leq_Q x_n$ for some $i \in
\Nb$. As $n$ is not a true stage, there is some $k > n$ such that $f(k) <
f(n)$, and therefore $n \notin T_m$ for all $m \geq k$. Let $m > k$ be such that
$P_m$ contains $q_j$ for some $j > i$. Then $x_n \leq_Q P_m$ by
Lemma~\ref{lem-Qproperties}~(ii), so we have that $q_i \leq_Q x_n \leq_Q
q_j$, contradicting that $(q_i)_{i \in \Nb}$ is a bad sequence.
\end{proof}

We now present our main reversals.

\begin{Theorem}\label{thm-FlatReversal}
The statement ``if $Q$ is a wqo, then $\upper(\Pf^\flat(Q))$ is Noetherian''
implies $\aca$ over $\rca$.
\end{Theorem}

\begin{proof}
Let $f \colon \Nb \imp \Nb$ be an injection. By Lemma~\ref{lem-ACAinjection},
it suffices to show that the range of $f$ exists. Let $P$ be the partial
order $P = \{x, y, z\}$ with $x <_P z$ and $x, z \mid_P y$, and let $Q =
\Xi_f(P,x)$.

We show that $\upper(\Pf^\flat(Q))$ is not Noetherian. Then by our
hypothesis $Q$ is not a wqo, and the existence of the range of $f$ follows from
Lemma~\ref{lem-ACAreversal}. To witness that $\upper(\Pf^\flat(Q))$ is not
Noetherian, we define a sequence $(E_s)_{s \in \Nb}$ of finite subsets of
$\Pf(Q)$ so that the corresponding sequence of effectively closed sets
$(F_s)_{s \in \Nb}$, given by $F_s = E_s\da^\flat$ for each $s \in \Nb$, is
descending but does not stabilize. Notice that in fact $(F_s)_{s \in \Nb}$ is
a sequence of basic closed sets.

For each $s \in \Nb$, let
\begin{gather*}
E_s = \{\bs{a}_s, \bs{b}_s\} \cup \{\bs{b}_n : n \in T_s\},
\intertext{where}
\bs{a}_s = \{x_s, y_s\} \cup \{y_n : n \in T_s\} \quad
\text{and} \quad \bs{b}_s = \{z_s\} \cup \{y_n : n \in T_s\}.
\end{gather*}

We need to show that $F_s \supsetneqq F_{s+1}$ for each $s \in \Nb$. As
$T_{s+1} \subseteq T_s \cup \{s\}$, by the definition of $E_s$ we always
have that $\{\bs{b}_n : n \in T_{s+1}\} \subseteq E_s$. Thus to prove the
inclusion, we focus on $\bs{a}_{s+1}$ and $\bs{b}_{s+1}$.

First suppose that $T_{s+1} \subsetneqq T_s \cup \{s\}$, and let $n_0$ be the
least element of $(T_s \cup \{s\}) \setminus T_{s+1}$. By the construction of
$Q$, $P_{s+1}$ is placed between $x_{n_0}$ and $z_{n_0}$, and therefore
$x_{s+1}, y_{s+1}, z_{s+1} <_Q z_{n_0}$. As argued in the proof of
Lemma~\ref{lem-Qproperties}, it must be that $f(s+1) < f(n_0)$ and $(\forall
k \in (n_0,s])(f(n_0) < f(k))$. Therefore $(\forall k \in [n_0,s])(f(s+1) <
f(k))$, and $s+1$ witnesses that no element in the interval $[n_0,s]$ is true.
This implies that $T_{s+1} \subseteq T_{n_0}$. We now see that
$E_s\da^\flat \supseteq E_{s+1}\da^\flat$: $\bs{a}_{s+1},
\bs{b}_{s+1} \leq_Q^\flat \bs{b}_{n_0}$ because $x_{s+1}, y_{s+1}, z_{s+1}
<_Q z_{n_0}$ and $\{y_n : n \in T_{s+1}\} \subseteq \{y_n : n \in T_{n_0}\}$,
and $\bs{b}_{n_0} \in E_s$ because either $n_0 = s$ or $n_0 \in T_s$.

We now show that $E_s\da^\flat \supsetneqq E_{s+1}\da^\flat$ by showing that
$\bs{b}_{n_0} \notin E_{s+1}\da^\flat$. This means that we need to show that $
\bs{b}_{n_0}
\nleq_Q^\flat \bs{a}_{s+1}$, $\bs{b}_{n_0} \nleq_Q^\flat \bs{b}_{s+1}$, and
$\bs{b}_{n_0} \nleq_Q^\flat \bs{b}_n$ for each $n \in T_{s+1}$. Notice that
$x_{s+1}, y_{s+1}, z_{s+1} <_Q z_{n_0}$, and if $n \in T_{s+1} \subseteq
T_{n_0}$, then $z_{n_0} \mid_Q y_n$ by Lemma~\ref{lem-Qproperties}~(i).
Hence $z_{n_0} \notin \bs{a}_{s+1}\da$ and $z_{n_0} \notin \bs{b}_{s+1}\da$.
As $z_{n_0} \in \bs{b}_{n_0}$, it follows that $\bs{b}_{n_0} \nleq_Q^\flat
\bs{a}_{s+1}$ and $\bs{b}_{n_0} \nleq_Q^\flat \bs{b}_{s+1}$. Now fix $n \in
T_{s+1}$, and note that $y_n \in \bs{b}_{n_0}$ because $T_{s+1} \subseteq
T_{n_0}$. However, $y_n \notin \bs{b}_n\da$ because $y_n \mid_Q
z_n$ by the definition of $P$, and $y_n \mid_Q y_\ell$ for all $\ell \in T_n$
by Lemma~\ref{lem-Qproperties}~(i). Thus $\bs{b}_{n_0} \nleq_Q^\flat
\bs{b}_n$.

Now suppose that $T_{s+1} = T_s \cup \{s\}$. Then obviously $\{y_n : n \in
T_{s+1}\}= \{y_n : n \in T_s\} \cup \{y_s\}$ and, since in this case
$P_{s+1}$ is placed immediately below $x_s$, we have $x_{s+1}, y_{s+1},
z_{s+1} <_Q x_s$. Thus $\bs{a}_{s+1}, \bs{b}_{s+1} \leq_Q^\flat
\bs{a}_s$, and so $E_s\da^\flat \supseteq E_{s+1}\da^\flat$. We show that
$E_s\da^\flat \supsetneqq E_{s+1}\da^\flat$ by showing that $\bs{a}_s \notin
E_{s+1}\da^\flat$. We already noticed that $x_{s+1}, y_{s+1}, z_{s+1} <_Q
x_s$. If $n \in T_{s+1}$, then either $n=s$, in which case $x_s \mid_Q y_n$
by the definition of $P$, or $n \in T_s$, in which case $x_s \mid_Q y_n$ by
Lemma~\ref{lem-Qproperties}~(i). This shows that $x_s \notin \bs{a}_{s+1}\da
\cup \bs{b}_{s+1}\da$ and thus (because $x_s \in \bs{a}_s$) that $\bs{a}_s
\nleq_Q^\flat \bs{a}_{s+1}$ and $\bs{a}_s \nleq_Q^\flat \bs{b}_{s+1}$. For $n
\in T_{s+1} = T_s \cup \{s\}$, we have that $\bs{a}_s \nleq_Q^\flat \bs{b}_n$
because $y_n \in \bs{a}_s$ but, as explained in the preceding paragraph, $y_n
\notin \bs{b}_n\da$.  Thus $\bs{a}_s \nleq_Q^\flat \bs{b}_n$ for each $n \in T_{s
+1}$ and therefore $\bs{a}_s \notin E_{s+1}\da^\flat$.

This completes the proof that $(F_s)_{s \in \Nb}$ witnesses that
$\upper(\Pf^\flat(Q))$ is not Noetherian.
\end{proof}

Notice that Theorem~\ref{thm-FlatReversal} gives an alternate reversal for
Theorem~\ref{thm-FlatWQOinACA} because the statement ``if $Q$ is a wqo, then
$\Pf^\flat(Q)$ is a wqo'' implies the statement ``if $Q$ is a wqo, then
$\upper(\Pf^\flat(Q))$ is Noetherian'' over $\rca$ by
Proposition~\ref{prop-WQOvsTopInRCA}. Thus we may see
Theorem~\ref{thm-FlatReversal} as a strengthening of the reversal in
Theorem~\ref{thm-FlatWQOinACA}.

\begin{Theorem}\label{thm-SharpReversal}
The statement ``if $Q$ is a wqo, then $\upper(\Pf^\sharp(Q))$ is Noetherian''
implies $\aca$ over $\rca$.
\end{Theorem}

\begin{proof}
Let $f \colon \Nb \imp \Nb$ be an injection. Let $P$ be the partial order $P
= \{x, y\}$ with $x \mid_P y$, and let $Q = \Xi_f(P,x)$. As in the proof of
Theorem~\ref{thm-FlatReversal}, it suffices to show that
$\upper(\Pf^\sharp(Q))$ is not Noetherian, and then appeal to the hypothesis,
Lemma~\ref{lem-ACAinjection}, and Lemma~\ref{lem-ACAreversal}.

To show that $\upper(\Pf^\sharp(Q))$ is not Noetherian, we define a sequence
$(E_s)_{s \in \Nb}$ of finite subsets of $\Pf(Q)$ so that the sequence of
effectively closed sets $(F_s)_{s \in \Nb}$, where $F_s = \bigcap_{t \leq s}
E_t\da^\sharp$ for each $s \in \Nb$, is descending but does not stabilize. We
define the sequence $(E_s)_{s \in \Nb}$ in stages along with sequences of
elements $(\bs{a}_s)_{s \in \Nb}$ and $(\bs{b}_s)_{s \in \Nb}$ from $\Pf(Q)$.
We ensure that $E_s$ is always a subset of
$\{\bs{a}_t : t \leq s\} \cup \{\bs{b}_t : t \leq s\}$ and always contains
$\bs{a}_s$ and $\bs{b}_s$. Among the sets in $E_s$,
$\bs{a}_s$ is the unique set containing $x_s$, and $\bs{b}_s$ is the unique
set containing $y_s$.

At stage $0$, let $\bs{a}_0 = \{x_0\}$, let $\bs{b}_0 = \{y_0\}$, and let
$E_0 = \{\bs{a}_0, \bs{b}_0\}$. At stage $s+1$, the definition of $E_{s+1}$
proceeds according to the construction of $Q$.
\begin{enumerate}[(i)]
\item If $T_{s+1} \subsetneqq T_s \cup \{s\}$ and $n_0$ is the least
    element of $(T_s \cup \{s\}) \setminus T_{s+1}$, then set $\bs{a}_{s+1}
    = \bs{b}_{n_0} \cup \{x_{s+1}\}$, $\bs{b}_{s+1} = \bs{b}_{n_0} \cup
    \{y_{s+1}\}$, and $E_{s+1} = (E_{n_0} \setminus \{\bs{a}_{n_0},
    \bs{b}_{n_0}\}) \cup \{\bs{a}_{s+1}, \bs{b}_{s+1}\}$.

\item If $T_{s+1} = T_s \cup \{s\}$, then set $\bs{a}_{s+1} = (\bs{a}_s
    \setminus \{x_s\}) \cup \{x_{s+1}\}$, $\bs{b}_{s+1} = (\bs{b}_s
    \setminus \{y_s\}) \cup \{y_{s+1}\}$, and $E_{s+1} = (E_s \setminus
    \{\bs{a}_s\}) \cup \{\bs{a}_{s+1}, \bs{b}_{s+1}\}$.
\end{enumerate}

The sequence $(F_s)_{s \in \Nb}$ is clearly descending by definition; we must
show that it is strictly descending. To do this, we identify some helpful
properties of the sequences $(\bs{a}_s)_{s \in \Nb}$, $(\bs{b}_s)_{s \in
\Nb}$, and $(E_s)_{s \in \Nb}$. First, observe that $\forall s(\bs{a}_s
\setminus \{x_s\} = \bs{b}_s \setminus \{y_s\})$ by an easy induction
argument.

\begin{NClaim}\label{claim-UEantichain}
$\forall s(\text{$\bigcup E_s$ is an antichain in $Q$})$.
\end{NClaim}

\begin{proof}[Proof of claim]
By $\Sigma^0_0$ induction on $s$. The case $s=0$ is clear because $x_0 
\mid_Q
y_0$. Consider $s+1$. First suppose that $E_s$ is defined according to~(i)
and that $n_0$ is the least element of $(T_s \cup \{s\}) \setminus T_{s+1}$.
By the induction hypothesis, $\bigcup E_{n_0}$ is an antichain. By the
construction of $Q$, $x_{s+1}$ and $y_{s+1}$ are placed immediately above
$x_{n_0} \in \bigcup E_{n_0}$ and hence are incomparable with the elements of
$\bigcup E_{n_0} \setminus \{x_{n_0}\}$. Thus $\bigcup E_{s+1} \subseteq
(\bigcup E_{n_0} \setminus \{x_{n_0}\}) \cup \{x_{s+1}, y_{s+1}\}$ (in fact
the reader can check that $\bigcup E_{s+1} = (\bigcup E_{n_0} \setminus
\{x_{n_0}\}) \cup \{x_{s+1}, y_{s+1}\}$) is an antichain.

Now suppose that $E_{s+1}$ is defined according to~(ii). By the induction
hypothesis, $\bigcup E_s$ is an antichain. By the construction of $Q$,
$x_{s+1}$ and $y_{s+1}$ are placed immediately below $x_s \in \bigcup E_s$
and hence are incomparable with the elements of $\bigcup E_s \setminus
\{x_s\}$. Thus $\bigcup E_{s+1} \subseteq (\bigcup E_s \setminus \{x_s\})
\cup \{x_{s+1}, y_{s+1}\}$ (again, in fact the reader can check that $\bigcup
E_{s+1} = (\bigcup E_s \setminus \{x_s\}) \cup \{x_{s+1}, y_{s+1}\}$) is an
antichain.
\end{proof}

\begin{NClaim}\label{claim-avoid}
$\forall s(\bs{a}_s \notin (E_s \setminus \{\bs{a}_s\})\da^\sharp \andd
\bs{b}_s \notin (E_s \setminus \{\bs{b}_s\})\da^\sharp)$.
\end{NClaim}

\begin{proof}[Proof of claim]
By $\Sigma^0_0$ induction on $s$. The case $s=0$ is clear. Consider $s+1$.
First suppose that $E_{s+1}$ is defined according to~(i), and let $n_0$ be
the least element of $(T_s \cup \{s\}) \setminus T_{s+1}$. Suppose that
$\bs{a}_{s+1} \leq_Q^\sharp \bs e$ for some $\bs e \in E_{n_0} \setminus
\{\bs{a}_{n_0}, \bs{b}_{n_0}\}$. Then $\bs e \subseteq \bs{a}_{s+1}\ua$. As
$\bs{a}_{n_0}$ is the only element of $E_{n_0}$ containing $x_{n_0}$, $x_{n_0}
\notin \bs e$ and so $\bs e \cup \{x_{n_0}\}$ is an antichain by
Claim~\ref{claim-UEantichain}; in particular, no element of $\bs e$ is $\geq_Q
x_{n_0}$.
However, $\bs{a}_{s+1}$ is $\bs{b}_{n_0} \cup \{x_{s+1}\}$, and $x_{s+1}
\geq_Q x_{n_0}$. It follows that $\bs e
\subseteq \bs{b}_{n_0}\ua$ and so $\bs{b}_{n_0} \in (E_{n_0} \setminus
\{\bs{b}_{n_0}\})\da^\sharp$. This contradiction to the induction hypothesis shows
that
$\bs{a}_{s+1} \notin (E_{s+1} \setminus \{\bs{a}_{n_0},\bs{b}_{n_0}\})\da^\sharp$.
Finally,
$\bs{a}_{s+1} \mid_Q^\sharp \bs{b}_{s+1}$ because $\bs{a}_{s+1}
\cup \bs{b}_{s+1}$ is an antichain by Claim~\ref{claim-UEantichain}, which
means that $x_{s+1} \notin \bs{b}_{s+1}\ua$ and $y_{s+1} \notin
\bs{a}_{s+1}\ua$. Thus $\bs{a}_{s+1} \notin (E_{s+1} \setminus
\{\bs{a}_{s+1}\})\da^\sharp$. A similar argument shows that $\bs{b}_{s+1}
\notin (E_{s+1} \setminus \{\bs{b}_{s+1}\})\da^\sharp$.

Now consider the case that $E_{s+1}$ is defined according to~(ii), and suppose
that
$\bs{a}_{s+1} \leq_Q^\sharp \bs e$ for some $\bs e \in E_s \setminus
\{\bs{a}_s\}$. Then $\bs e \subseteq \bs{a}_{s+1}\ua$. However, since
$\bs{a}_{s+1}$ is given by replacing $x_s$  with $x_{s+1}$ in $\bs{a}_s$ and
$x_{s+1}$ is
placed immediately below $x_s$, any $z \in \bs e$ that is $\geq_Q
x_{s+1}$ is also $\geq_Q x_s$, and therefore $\bs e \subseteq \bs{a}_s\ua$ as
well. So $\bs{a}_s \in (E_s \setminus \{\bs{a}_s\})\da^\sharp$, which
contradicts the induction hypothesis. Thus $\bs{a}_{s+1} \notin (E_s \setminus
\{\bs{a}_s\})\da^\sharp$.
Since $\bs{a}_{s+1} \mid_Q^\sharp \bs{b}_{s+1}$, as argued in the previous case,
we again have that
$\bs{a}_{s+1} \notin (E_{s+1} \setminus \{\bs{a}_{s+1}\})\da^\sharp$.
A similar argument shows that $\bs{b}_{s+1} \notin (E_{s+1} \setminus
\{\bs{b}_{s+1}\})\da^\sharp$.
\end{proof}

\begin{NClaim}\label{claim-persist}
$(\forall s)(\forall i \leq s)(\bs{a}_s \in E_i \da^\sharp \andd \bs{b}_s \in
E_i \da^\sharp)$.
\end{NClaim}

\begin{proof}[Proof of claim]
By $\Sigma^0_0$ induction on $s$. The case $s=0$ is clear. Consider $s+1$.
First suppose that $E_{s+1}$ is defined according to~(i), and let $n_0$ be
the least element of $(T_s \cup \{s\}) \setminus T_{s+1}$. By the induction
hypothesis for $n_0$, $(\forall i \leq n_0)(\bs{b}_{n_0} \in E_i
\da^\sharp)$. Since $\bs{a}_{s+1}\leq_Q^\sharp \bs{b}_{n_0}$ and
$\bs{b}_{s+1} \leq_Q^\sharp \bs{b}_{n_0}$, it suffices to show that $(\forall
i \leq s)(\bs{b}_{n_0} \in E_i \da^\sharp)$. By definition, $\bs{b}_{n_0} \in
E_{n_0}$; and if $\bs{b}_{n_0} \in E_i$ and $E_{i+1}$ is defined according
to~(ii), then $\bs{b}_{n_0} \in E_{i+1}$. So if $\bs{b}_{n_0} \notin E_{i+1}$
for some $i+1 \in (n_0, s]$, it must be because $E_{i+1}$ is defined
according to~(i) and the least element $n_1$ of $(T_i \cup \{i\}) \setminus
T_{i+1}$ is less than $n_0$. Then $f(i+1) < f(n_1) < f(n_0)$ because $i+1$ is
the least number witnessing that $n_1$ is not true, contradicting that $n_0
\in T_s \cup \{s\}$. Hence $(\forall i \leq s)(\bs{b}_{n_0} \in
E_i\da^\sharp)$.

Now suppose that $E_{s+1}$ is defined according to~(ii). By the induction
hypothesis, $(\forall i \leq s)(\bs{a}_s \in E_i \da^\sharp)$. As
$\bs{a}_{s+1} \leq_Q^\sharp \bs{a}_s$ and $\bs{b}_{s+1} \leq_Q^\sharp
\bs{a}_s$ (because in this case $x_{s+1}, y_{s+1} <_Q x_s$), it follows that
$(\forall i \leq s+1)(\bs{a}_{s+1} \in E_i \da^\sharp \andd \bs{b}_{s+1} \in
E_i \da^\sharp)$ as well.
\end{proof}

We can now show that $(F_s)_{s \in \Nb}$ is strictly descending. Consider $s
\in \Nb$. Suppose that $E_{s+1} = E_{n_0} \setminus (\{\bs{a}_{n_0},
\bs{b}_{n_0}\}) \cup \{\bs{a}_{s+1}, \bs{b}_{s+1}\}$ is defined according
to~(i), where $n_0$ is the least element of $(T_s \cup \{s\}) \setminus
T_{s+1}$. Then $\bs{b}_{n_0} \in \bigcap_{t \leq s}E_t\da^\sharp = F_s$ as
shown in the proof of Claim~\ref{claim-persist}. However, $\bs{a}_{s+1}
<_Q^\sharp \bs{b}_{n_0}$ and $\bs{b}_{s+1} <_Q^\sharp \bs{b}_{n_0}$ because
neither $x_{s+1}$ nor $y_{s+1}$ is above any element of $\bs{b}_{n_0}$ by
Claim~\ref{claim-UEantichain}, and $\bs{b}_{n_0}$ is not $\leq_Q^\sharp$ any
element of $E_{n_0} \setminus \{\bs{a}_{n_0}, \bs{b}_{n_0}\}$ by
Claim~\ref{claim-avoid}. Thus $\bs{b}_{n_0} \notin E_{s+1}\da^\sharp$, so
$\bs{b}_{n_0} \notin F_{s+1}$.

Finally, suppose that $E_{s+1} = (E_s \setminus \{\bs{a}_s\}) \cup
\{\bs{a}_{s+1}, \bs{b}_{s+1}\}$ is defined according to~(ii). Then by
Claim~\ref{claim-persist}, $\bs{a}_s \in \bigcap_{t \leq s}E_t\da^\sharp =
F_s$. On the other hand, since neither $x_{s+1}$ nor $y_{s+1}$ is above any
element of $\bs{a}_s$, $\bs{a}_{s+1} <_Q^\sharp \bs{a}_s$ and $\bs{b}_{s+1}
<_Q^\sharp \bs{a}_s$, while $\bs{a}_s$ is not $\leq_Q^\sharp$ any element of
$E_s \setminus \{\bs{a}_s\}$ by Claim~\ref{claim-avoid}. Thus $\bs{a}_s
\notin E_{s+1}\da^\sharp$, and so $\bs{a}_s \notin F_{s+1}$.

This completes the proof that $(F_s)_{s \in \Nb}$ witnesses that
$\upper(\Pf^\flat(Q))$ is not Noetherian.
\end{proof}

\begin{Theorem}~\label{thm-BigEquiv}
The following are equivalent over $\rca$.
\begin{enumerate}[(i)]
\item $\aca$.
\item If $Q$ is a wqo, then $\alex(\Pf^\flat(Q))$ is Noetherian.
\item If $Q$ is a wqo, then $\upper(\Pf^\flat(Q))$ is Noetherian.
\item If $Q$ is a wqo, then $\upper(\Pf^\sharp(Q))$ is Noetherian.
\item If $Q$ is a wqo, then $\upper(\Ps^\flat(Q))$ is Noetherian.
\item If $Q$ is a wqo, then $\upper(\Ps^\sharp(Q))$ is Noetherian.
\end{enumerate}
\end{Theorem}

\begin{proof}
That (i) implies (ii) and (iii) is Theorem~\ref{thm-FlatInACA}. That (i)
implies (iv) is Corollary~\ref{thm-SharpInACA}. That (i) implies (v) is
Theorem~\ref{thm-UnctblFlatInACA}. That (i) implies (vi) is
Theorem~\ref{thm-UnctblSharpInACA}. That (ii), (iii), and (v) imply (i) is
Theorem~\ref{thm-FlatReversal}. For (ii), use also
Proposition~\ref{prop-WQOvsTopInRCA}, and for (v), use also
Theorem~\ref{thm-UnctblNoethImpliesCtblNoeth}~(i). That (iv) and (vi) imply
(i) is Theorem~\ref{thm-SharpReversal}. For (vi), use also
Theorem~\ref{thm-UnctblNoethImpliesCtblNoeth}~(ii).
\end{proof}

Upon hearing the third author speak about the results contained in this
paper, Takashi Sato asked whether the converses of the statements
in Theorem~\ref{thm-BigEquiv} hold, and, for those that do, what system is 
needed to prove them.  First notice that $\upper(\Pf^\flat(Q))$ can be Noetherian
without $Q$ being a wqo.  Indeed, if $Q$ is an infinite antichain, then all closed
sets in $\upper(\Pf^\flat(Q))$ are finite and thus $\upper(\Pf^\flat(Q))$ is
Noetherian.  Nevertheless, $\rca$ easily proves that if $\upper(\Pf^\flat(Q))$ is
Noetherian, then $Q$ is well-founded.  The next proposition shows that the other
converses are provable in $\rca$.

\begin{Proposition}[$\rca$]
Let $Q$ be a quasi-order.  If $\alex(\Pf^\flat(Q))$, $\upper(\Pf^\sharp(Q))$,
$\upper(\Ps^\flat(Q))$, or $\upper(\Ps^\sharp(Q))$ is Noetherian, then $Q$
is a wqo.
\end{Proposition}

\begin{proof}
We prove the contrapositive.  The result for $\alex(\Pf^\flat(Q))$ follows from
the fact that if $Q$ is not wqo then $\Pf^\flat(Q)$ is not wqo (because
$q \mapsto \{q\}$ embeds $Q$ into $\Pf^\flat(Q)$) and from the second part
of Proposition~\ref{prop-WQOvsTopInRCA}.

For the other spaces, first fix a bad sequence $(q_i)_{i \in \Nb}$ of elements 
of $Q$.

To see that $\upper(\Ps^\flat(Q))$ is not Noetherian, let $F_n = \bigcap_{i
\leq n} \{Q \setminus (q_i\ua)\}\da^\flat$.  The sequence $(F_n)_{n \in \Nb}$
is a non-stabilizing descending sequence of effectively closed sets as
witnessed by $\{q_{n+1}\} \in F_n \setminus F_{n+1}$.

To see that $\upper(\Pf^\sharp(Q))$ is not Noetherian, let $H_n = \{q_i : i \leq
n\}\da^\sharp$.  The sequence $(H_n)_{n \in \Nb}$ is a non-stabilizing
descending sequence of effectively closed sets as witnessed by $\{q_i : i
\leq n\} \in H_n \setminus H_{n+1}$.

The result for $\upper(\Ps^\sharp(Q))$ follows easily from the result for 
$\upper(\Pf^\sharp(Q))$ and
Theorem~\ref{thm-UnctblNoethImpliesCtblNoeth}~(ii).
\end{proof}

\bibliography{noetherian}

\end{document}